\documentclass[11pt,english]{amsart}
\PassOptionsToPackage{obeyFinal}{todonotes}
\usepackage[T1]{fontenc}
\usepackage[latin9]{inputenc}
\usepackage{babel}
\usepackage{todonotes}
\usepackage{amsthm}
\usepackage{amssymb}
\usepackage{graphicx,xcolor}
\usepackage[letterpaper]{geometry}
\usepackage{pinlabel}
\usepackage[pdfusetitle,
 bookmarks=true,bookmarksnumbered=false,bookmarksopen=false,
 breaklinks=false,pdfborder={0 0 1},backref=false,colorlinks=false]
 {hyperref}
\hypersetup{
 colorlinks=true,linkcolor=teal,citecolor=purple,linktocpage=true}

\usepackage{comment}

\makeatletter
\numberwithin{equation}{section}
\numberwithin{figure}{section}
\usepackage[inline]{enumitem}
\theoremstyle{plain}
\newtheorem{thm}{\protect\theoremname}[section]
\theoremstyle{definition}
\newtheorem{defn}[thm]{\protect\definitionname}
\theoremstyle{plain}
\newtheorem{lem}[thm]{\protect\lemmaname}
\theoremstyle{plain}
\newtheorem{cor}[thm]{\protect\corollaryname}
\theoremstyle{plain}
\newtheorem{conj}[thm]{Conjecture}
\theoremstyle{plain}
\newtheorem{question}[thm]{Question}
\theoremstyle{remark}
\newtheorem{rem}[thm]{\protect\remarkname}
\newtheorem*{rem*}{\protect\remarkname}

\usepackage{xcolor}

\@ifundefined{date}{}{\date{}}
\usepackage{amsmath, tikz-cd, amssymb, placeins}
\usepackage{listings}
\DeclareRobustCommand*\cal{\@fontswitch\relax\mathcal}

\usetikzlibrary{calc}
\usetikzlibrary{decorations.pathmorphing}

\tikzset{curve/.style={settings={#1},to path={(\tikztostart)
    .. controls ($(\tikztostart)!\pv{pos}!(\tikztotarget)!\pv{height}!270:(\tikztotarget)$)
    and ($(\tikztostart)!1-\pv{pos}!(\tikztotarget)!\pv{height}!270:(\tikztotarget)$)
    .. (\tikztotarget)\tikztonodes}},
    settings/.code={\tikzset{quiver/.cd,#1}
        \def\pv##1{\pgfkeysvalueof{/tikz/quiver/##1}}},
    quiver/.cd,pos/.initial=0.35,height/.initial=0}

\tikzset{tail reversed/.code={\pgfsetarrowsstart{tikzcd to}}}
\tikzset{2tail/.code={\pgfsetarrowsstart{Implies[reversed]}}}
\tikzset{2tail reversed/.code={\pgfsetarrowsstart{Implies}}}
\tikzset{no body/.style={/tikz/dash pattern=on 0 off 1mm}}

\makeatother

\providecommand{\corollaryname}{Corollary}
\providecommand{\definitionname}{Definition}
\providecommand{\lemmaname}{Lemma}
\providecommand{\theoremname}{Theorem}
\providecommand{\remarkname}{Remark}

\newlist{eqenumerate}{enumerate}{1}
\setlist[eqenumerate]{
  leftmargin=3\parindent,
  align=left,
  labelwidth=3\parindent,
  labelsep=0pt
}

\newcommand{\eqitem}{
  \refstepcounter{equation}
  \item[(\theequation)]
}

\DeclareMathOperator{\GL}{GL}
\DeclareMathOperator{\partialext}{\partial_{\mathrm{ext}}}
\DeclareMathOperator{\partialint}{\partial_{\mathrm{int}}}

\begin{document}
\title{An irreducible real projective plane in the 4-sphere}

\author{Mark Hughes}
\address{Brigham Young University\\Provo, UT 84602 USA}
\email{hughes@math.byu.edu}

\author{Seungwon Kim}
\address{Sungkyunkwan University\\Suwon, Gyeonggi, 16419 Republic of Korea}
\email{seungwon.kim@skku.edu}

\author{Maggie Miller}
\address{Department of Mathematics, University of Texas at Austin, TX 78712 USA}
\email{maggie.miller.math@gmail.com}

\author{Gheehyun Nahm}
\address{Department of Mathematics, Princeton University, Princeton, New Jersey
08544, USA}
\email{gn4470@math.princeton.edu}

\thanks {MH was supported by the NSF (DMS-2522496).
SK was supported by National Research Foundation of Korea (NRF) grants funded by the Korean government (MSIT) (No.\ 2022R1C1C2004559). MM was partially supported by a Packard Fellowship for Science and Engineering, a Sloan Research Fellowship, and NSF grant DMS-2404810.
GN was partially supported by the ILJU Academy and Culture Foundation, the Simons collaboration \emph{New structures in low-dimensional topology}, and a Princeton Centennial Fellowship.}

\begin{abstract}
    We construct an irreducible embedded projective plane in $S^4$.
    This gives a counterexample to the Kinoshita conjecture and answers Problem~4.37 of the K3 problem list. 
    Moreover, we answer both Questions~(i)~and~(ii) of Problem~4.37:
    (i) the connected sum $R\#R$ is a Klein bottle in $S^4$ with extremal normal Euler number that does not admit an unknotted projective plane summand,
    and (ii) we show that our projective plane $R$ is irreducible by showing that the peripheral map $\pi_1 (\partial (S^4\setminus\mathring{N}(R)))\to \pi_1 (S^4 \setminus \mathring{N}(R))$ has kernel of order $2$.
\end{abstract}

\maketitle

\section{Introduction}\label{sec:intro}

In this paper, we construct the first \emph{irreducible} embedding of the projective plane in $S^4$.

\begin{thm}\label{thm:kinoshita}
There exists a smoothly embedded projective plane in $S^4$ which cannot be topologically decomposed as the connected sum of a knotted $2$-sphere with an unknotted projective~plane.
\end{thm}

We say a surface smoothly embedded in $S^4$ is {\emph{reducible}} if it admits an unknotted summand that is not a $2$-sphere. (Note that this is distinct from ``prime,'' and agrees with the definitions of e.g.\ \cite{MR769824,MR1309400}.) Theorem \ref{thm:kinoshita} serves as a counterexample to the following well-known conjecture, which appears in the updated K3 problem list \cite{baykur295k3} as Problem 4.37.

\begin{conj}[{Kinoshita conjecture, \cite[Problem 4.37]{baykur295k3}}]
    Every knotted projective plane in $S^4$ is reducible.
\end{conj}

When a projective plane in $S^4$ is reducible, it is also common to say it is of {\emph{Kinoshita type}} (see e.g.\ \cite{yoshikawa86}). So to rephrase Theorem \ref{thm:kinoshita}, we prove that not every projective plane in $S^4$ is of Kinoshita type.
While the conjecture is now commonly named after the work of Kinoshita \cite{kin61}, it seems to have first appeared as the ``Kinoshita conjecture'' in published 
work of Katanaga--Saeki \cite[Remark 3.7]{katanaga1998embeddings}, who said that the problem was well-known to topologists in Japan by 1997. For example, it appears in a 1986 conference proceedings 
paper of Yoshikawa \cite[Question 1.2]{yoshikawa86}.
An informal version of this problem appeared as early as 1975 in work of Price--Roseman \cite{PriceRoseman1975}, who said ``we don't
know of any embedded projective plane that is not [of Kinoshita type], although they surely must exist.''

In addition to disproving the Kinoshita conjecture, our example also answers a separate question included in the remarks of \cite[Problem 4.37]{baykur295k3},
which asks whether there is an irreducible Klein bottle in $S^4$ with normal Euler number $\pm 4$. Previous examples of irreducible Klein bottles (such as the original examples constructed by Yoshikawa \cite{yoshikawa1998order}) have normal Euler number zero. (A Klein bottle in $S^4$ must have normal Euler number $-4,0$, or $4$ by the Whitney--Massey theorem \cite{MR250331}.)

\begin{thm}[{Answer to \cite[Problem 4.37, Question (i)]{baykur295k3}}]\label{thm:klein}
    There exists a topologically irreducible, smoothly embedded Klein bottle in $S^4$ with normal Euler number $\pm 4$.
\end{thm}

We do not specify the sign in Theorem \ref{thm:klein} since mirroring $S^4$ preserves irreducibility of the surface. Again, this surface is smoothly embedded but the obstruction to reducibility is locally flat. In fact, the Klein bottle in Theorem \ref{thm:klein} is a connected sum of two copies of the projective plane of Theorem \ref{thm:kinoshita}.

There are many known examples of irreducible surfaces in $S^4$, including tori \cite{MR436150}, Klein bottles of normal Euler number zero \cite{yoshikawa1998order,yoshikawa86}, and links of projective planes \cite{MR1309400}.
For further examples and discussion see e.g.\ \cite{MR474320, MR591977,  MR635591, MR705232, MR769824, Livingston1988,  lidman2025stably,meier2025klein}.

In both Theorem \ref{thm:kinoshita} and Theorem \ref{thm:klein}, the obstruction to reducibility is obtained by considering the {\emph{peripheral subgroup}}, i.e.\ the image of $\pi_1(\partial( S^4\setminus \mathring{N}(F)))\to\pi_1(S^4\setminus\mathring{N}(F))$ for a surface $F$.
By an application of the Seifert--van Kampen theorem,  a reducible projective plane has peripheral subgroup isomorphic to $\mathbb{Z}/2\mathbb{Z}$ (see Lemma \ref{lem:peripheral-obstruction}).
By a similar argument, a reducible Klein bottle must also have finite peripheral subgroup (Lemma \ref{lem:reducible-klein}).
We show that the projective plane $R$ constructed in Theorem \ref{thm:kinoshita} has peripheral subgroup $(\mathbb{Z}/2\mathbb{Z})^2$ and that $R\#R$ has peripheral subgroup of infinite order.

Katanaga--Saeki \cite[Remark~3.7]{katanaga1998embeddings} conjectured that every knotted projective plane in $S^4$ must have peripheral subgroup $\mathbb{Z}/2\mathbb{Z}$,
and this question is also asked in \cite[Problem 4.37, Question~(ii)]{baykur295k3}.
The proof of Theorem \ref{thm:kinoshita} thus disproves Katanaga--Saeki's conjecture and answers \cite[Problem 4.37, Question (ii)]{baykur295k3}.

\begin{cor}[{Answer to \cite[Problem 4.37, Question (ii)]{baykur295k3}}]\label{cor:peripheral}
There exists a smoothly embedded projective plane $R$ in $S^4$ whose peripheral subgroup has order four.
\end{cor}

The texts of \cite[Remark~3.7]{katanaga1998embeddings} and \cite[Problem 4.37, Question (ii)]{baykur295k3} phrase this in terms of the order of the kernel of the map $\pi_1(\partial(S^4\setminus\mathring{N}(R)))\to\pi_1(S^4\setminus R)$,
but since $\pi_1(\partial N(R))\cong Q_8$ has finite order the size of this kernel is determined by the order of the peripheral subgroup.
\begin{rem}
The meridian of the projective plane $R$ of Theorem \ref{thm:kinoshita} and Corollary \ref{cor:peripheral} has order two, so the existence of $R$ also negatively answers \cite[Question 1.3]{yoshikawa86}: ``If the order of the meridian of a projective plane $P$ in $\mathbb{R}^4$ is two, then must $P$ be of Kinoshita type?'' (paraphrased from Japanese).
\end{rem}

Price--Roseman \cite[Question 18]{PriceRoseman1975} also asked the following question about peripheral subgroup, to which our examples notably do {\emph{not}} provide an answer.
\begin{question}
    Does the meridian of a locally flat projective plane in $S^4$ always have order two in the complement of the projective plane?
\end{question}

The counterexamples we construct to prove Theorems \ref{thm:kinoshita} and \ref{thm:klein} are smoothly embedded, but the obstruction from the peripheral subgroup holds in the topological category. Thus, we are left with the following relevant question.

\begin{question}
    Does there exist a smoothly embedded projective plane in $S^4$ that is topologically but not smoothly reducible?
\end{question}

\subsection*{A history of spinning}

The projective plane of Theorem \ref{thm:kinoshita} is constructed using a variation of {\emph{spinning}} classical knots in $S^3$, a common construction of surfaces in $S^4$ introduced by Artin \cite{artin1925isotopie}. Artin's spun knots are 2-spheres, as are the generalizations of Zeeman \cite{zeeman1965twisting}, Fox \cite{Fox1966}, and Litherland \cite{litherland1979deforming} of {\emph{deform-spun knots}}, in which a choice of symmetry of the classical knot determines the resulting 2-sphere. Price--Roseman \cite{PriceRoseman1975} introduced the idea of using an orientation-reversing symmetry, with the effect of producing a knotted projective plane in $S^4$.

Kamada \cite{kamada1990deform} showed that deform-spinning a strongly invertible knot using the strong inversion composed with any number of ``twists'' or ``rolls'' (as in \cite{litherland1979deforming}) produces a reducible projective plane. Kamada \cite{kamada1992projective} further showed that deform-spinning a torus knot or hyperbolic knot with {\emph{any}} choice of orientation-reversing diffeomorphism always produces a reducible projective plane. While we do not phrase our example exactly in this language, we will see later (Remark \ref{rem:spinning}) that our projective plane is obtained by deform-spinning a satellite knot, specifically the connected sum $\#_4 T_{3,4}$. The actual deformation relies heavily on the fact that the classical knot is not prime, and is also not a composition of a strong inversion with twists and rolls.

\subsection*{Conventions}\leavevmode
\begin{enumerate}[leftmargin=*]
\item All constructions are done in the smooth category. The tools used in Theorems~\ref{thm:kinoshita} and \ref{thm:klein} obstruct locally flat decomposition.
\item For convenience of notation, we henceforth write $C_n$ to denote the cyclic group of order $n$ (including the case $n=\infty$), and $C_n\{t\}$ to denote the cyclic group of order $n$ with $t$ as the generator.
\item All group operations are written multiplicatively.
\item We choose a standard identification of $S^1$ with $[0,1]/(1\sim 0)$, and similarly typically include a subscript to distinguish separate copies of the circle (e.g.\ $S^1_a, S^1_b$).
\item We compose paths and loops from left to right.
\item We write the pushout of group homomorphisms $H\to G_1, H\to G_2$ as $G_1\ast_H G_2$.
\end{enumerate}

\subsection*{Organization}
In Section \ref{sec:construction}, we go over the construction of a projective plane $R$. In Section \ref{sec:proof}, we show that $R$ and $R\#R$ are irreducible, proving Theorems \ref{thm:kinoshita} and \ref{thm:klein}.

\subsection*{Acknowledgements}
GN thanks Peter Ozsv\'{a}th for his continual support. 
The authors acknowledge the use of ChatGPT, Cursor, and Gemini during
initial exploration and computation. 
All final computations were performed and verified by the authors without the use of AI.
We thank Ikhan Choi for help finding a print reference unavailable online.

\section{\label{sec:construction}The construction}

In this section we construct an irreducible projective plane in $S^4$. 
We view $S^{4}$ as the union of
the \emph{southern hemisphere} $B_S^4$,
a tubular neighborhood $S^3 \times [-1,1]$ of the equatorial $S^3$,
and the \emph{northern hemisphere} $B_N^4$.
Here, $\partial B_{S}^{4}$ (resp.\ $\partial B_{N}^{4}$) is identified with $S^{3}\times\{-1\}$ (resp.\ $S^{3} \times \{1\}$).

\subsection{Half-spinning}

Consider the standard solid
torus $S^{1}\times D^{2}\subset S^{3}$; this gives rise to $S^{1}\times D^{2}\times[-1,1]\subset S^{3}\times[-1,1]$.
Let us identify $S^{1}=[0,1]/(1\sim0)$. Let $\varphi_{\theta}:D^{2}\times[-1,1]\to D^{2}\times[-1,1]$
be rotation by $\theta$ about the axis $\{\mathbf{0}\}\times[-1,1]$, in the direction indicated in Figure \ref{fig:pq-link}(a).
\begin{defn}[Half-spin]\label{def:halfspin}
Given a subset $Y\subset D^{2}\times[-1,1]$ such that $\varphi_{\pi}(Y)=Y$ setwise,
the \emph{half-spin} of $Y$ (in $S^{1}\times D^{2}\times[-1,1]$)
is
\[
H_Y:=\bigcup_{\theta\in[0,1]}\{\theta\}\times\varphi_{\pi\theta}(Y)\subset S^{1}\times D^{2}\times[-1,1].
\]

Given an arc $a\subset D^{2}\times[-1,1]$ such that its two endpoints
are on $D^{2}\times\{-1\}$ and $\varphi_{\pi}(a)=a$, the half-spin
of $a$ in $S^{1}\times D^{2}\times[-1,1]$ viewed as a subset of
$S^{3}\times[-1,1]$ is a M\"obius band whose boundary is an unknot
in $S^{3}\times\{-1\}$. The \emph{half-spin of $a$ in $S^{4}$} 
is the projective plane in $S^{4}$ obtained by capping this M\"obius band off with a boundary-parallel
$D^{2}$ in $B_{S}^{4}$.
\end{defn}

\begin{rem}[Normal Euler number of a half-spin]\label{rem:euler}
The normal Euler number of the projective plane that is the half-spin of an arc $a$ is either $\pm2$, with sign depending on the orientations of $S^1\times D^2\times[-1,1]$ and $B^4_S$. We could negate the normal Euler number by mirroring $S^4$ --- or reversing the direction of the rotation $\varphi_\theta$. Our choice of rotation direction (and hence normal Euler number) is arbitrary, but for the sake of being concrete we will keep the direction fixed.
\end{rem}

\begin{figure}[h]
\begin{centering}
\includegraphics{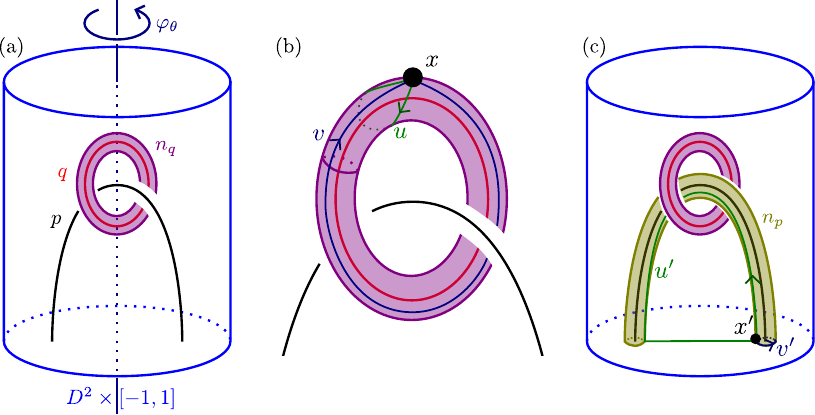}
\par\end{centering}
\caption{\label{fig:pq-link}(a): The half-spin of  $p \sqcup q \subset D^2 \times [-1,1]$ is the link $P\sqcup Q\subset S^{4}$ of an unknotted projective plane $P$ and an unknotted Klein bottle $Q$. 
The half-spin of the solid torus $n_q \supset q$ is a tubular neighborhood $N(Q)$ of $Q$ in $S^4$.
(b): A point $x\in \partial n_q \subset  D^{2}\times[-1,1]$ and loops $u,v$ on $n_q$ based at $x$ that give rise to loops in $\partial N(Q)$ based at $\overline{x}$.
(c): A change of basepoint from (b) to a basepoint $x'\in D^2 \times \{-1\}$.
The union of the southern hemisphere $B_S^4$ and the half-spin of $n_p \supset p$ is a tubular neighborhood $N(P)$ of $P$ in $S^4$.
The loops $u',v'$ are on $\partial n_p$.}
\end{figure}

Consider the half-spin of $p\sqcup q$ of Figure~\ref{fig:pq-link}~(a).
After capping off, this gives a link $P\sqcup Q\subset S^4$ of an unknotted projective plane $P$ and an unknotted Klein bottle $Q$. (As mentioned in Remark~\ref{rem:euler}, the normal Euler number of $P$ is $\pm 2$. The normal Euler number of $Q$ is evidently zero, since $Q$ bounds a solid Klein bottle obtained by half-spinning a disk bounded by $q$.) 
Our irreducible projective plane $R$ in $S^{4}$ will be obtained from a \emph{generalized knot surgery} on $Q$,
i.e.\ we will construct a $4$-manifold $W$ such that $\partial W=\partial N(Q)$ and $(S^{4}\setminus\mathring{N}(Q))\cup_{\partial N(Q)}W\cong S^{4}$.
The irreducible projective plane $R$ will be the image of $P$ in $(S^{4}\setminus\mathring{N}(Q))\cup W\cong S^{4}$.

\begin{rem}[Orientation convention]
In the construction of $W$, which is itself a union of key pieces, we will take orientable manifolds to be unoriented.  The final manifold $(S^4\setminus \mathring{N}(Q)) \cup W$ will then inherit an orientation from $S^4\setminus\mathring{N}(Q)$.
\end{rem}

\subsection{The manifold $X$}\label{sec:X}
Before building $W$, we will define a manifold $X$ with two boundary components. (Then $W$ will be obtained by filling one boundary component of $X$.)

Let $A:=D^{2}\setminus\bigsqcup_{i=1}^{4}\mathring{D_{i}^{2}}$ (see Figure~\ref{fig:a-disk})
and let $\beta$ be the 4-strand braid $\beta:=\sigma_{1}^{2}\sigma_{3}\sigma_{2}\sigma_3^{-1}\sigma_{1}^{-2}\in B_{4}$.
Then $\beta$ induces a diffeomorphism on $A$ by isotopy extension, which we also denote as $\beta:A\to A$. Our conventions are as follows.
\begin{enumerate}[label=(\roman*)]
    \item The braid group $B_4$ acts on $A$ on the left; i.e.,\ in the action of $\beta$ on $A$, we apply $\sigma_{1}^{-2}$ first, then $\sigma_{3}^{-1}$, and so on.
    \item \label{enu:action-gamma}The braid generator $\sigma_i \in B_4 $ acts on $\pi_1 (A,y)$ as follows: $\sigma_{i}\gamma_{i}=\gamma_{i}\gamma_{i+1}\gamma_{i}^{-1}$, $\sigma_{i}\gamma_{i+1}=\gamma_{i}$, and $\sigma_{i}\gamma_{j}=\gamma_{j}$ for $j\neq i,i+1$.
    The action of $\sigma_i ^{-1}$ is as follows: $\sigma_{i}^{-1}\gamma_{i}=\gamma_{i+1}$, $\sigma_{i}^{-1}\gamma_{i+1}=\gamma_{i+1}^{-1}\gamma_{i}\gamma_{i+1}$,
    and $\sigma_{i}^{-1}\gamma_{j}=\gamma_{j}$ for $j\neq i,i+1$.
    \item \label{enu:action-rho}While we isotope the disks $D_i ^2$, we do not rotate them, and so in particular $\beta$ permutes the $y_i$'s. Also, the braid word $\sigma_i$ acts on the paths $\rho_{j}$'s as follows (up to homotopy rel.~$\partial$): $\sigma_{i}\rho_{i}=\gamma_{i}\rho_{i+1}$, $\sigma_{i}\rho_{i+1}=\rho_{i}$, and $\sigma_{i}\rho_{j}=\rho_{j}$ for $j\neq i,i+1$.
    The action of $\sigma_i ^{-1}$ is as follows  (up to homotopy rel.~$\partial$): $\sigma_{i}^{-1}\rho_{i}=\rho_{i+1}$, $\sigma_{i}^{-1}\rho_{i+1}=\gamma_{i+1}^{-1}\rho_{i}$,
    and $\sigma_{i}^{-1}\rho_{j}=\rho_{j}$ for $j\neq i,i+1$.
\end{enumerate}

Let $r:D^2 \to D^2 $ be the horizontal reflection (Figure~\ref{fig:a-disk}); this restricts to a diffeomorphism of $A$ which we also denote as $r$.
We will also write $r$ to denote the restriction to $S^1$, i.e.\ $z\mapsto 1-z$ where $S^1=[0,1]/(1 \sim0)$. 

\begin{rem}\label{rem:transitive}
Denoting the induced
permutation $\{1,2,3,4\}\to\{1,2,3,4\}$ also as $\beta$, we find $\beta$ is the transposition $(2\ 4)$.
Similarly, $r$ induces the permutation $r:\{1,2,3,4\}\to\{1,2,3,4\}$ be $i\mapsto5-i$. Then $r\beta$ is the cycle $(1\ 4\ 3\ 2)$; in particular, $r\beta$ acts transitively on $\{1,2,3,4\}$.
\end{rem}

\begin{rem}[Choice of the braid $\beta$]

Denote as $r\beta:\pi_1 (A,y)\to \pi_1 (A,y)$ the map induced by the diffeomorphism $r\circ \beta: (A,y)\to (A,y)$.
A key property that the braid $\beta$ satisfies is that
the loop $\partial D^2 = \gamma_1 \gamma_2 \gamma_3 \gamma_4 $ is nontrivial in the quotient
\[
\pi_1 (A,y) /\left\langle \left\langle \gamma^{-1} (r\beta (\gamma )) \ \forall\gamma\in \pi_1 (A,y)\right\rangle \right\rangle.
\]
We use this fact in Steps 3 and 4 of the proof of Theorem~\ref{thm:kinoshita}.

\end{rem}

\begin{figure}[h]
\begin{centering}
\includegraphics{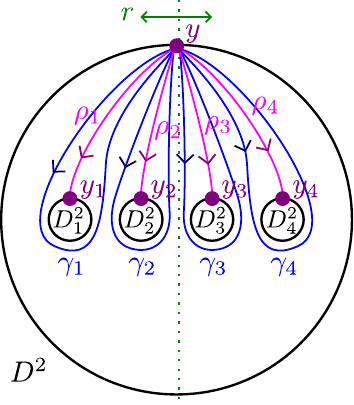}
\par\end{centering}
\caption{\label{fig:a-disk}Points $y$ and $y_i$, paths $\rho_i$, and loops $\gamma_i$ on $A:=D^{2}\setminus\sqcup_{i=1}^{4}\mathring{D_{i}^{2}}$, and the reflection $r:A\to A$. Note that $r(y_i ) = y_{5-i}.$}
\end{figure}

The manifold $X$ will be defined as a mapping torus; let us specify our conventions.
\begin{defn}[Mapping torus]
The \emph{mapping
torus} of a map $\varphi:Y\to Y$ is 
\[
S^1 \times _{\varphi} Y:=([0,1]\times Y)/(1,y)\sim(0,\varphi(y)).
\]
\end{defn}

\begin{rem}[Half-spin as a mapping torus]\label{rem:half-spin-mapping-torus}
For subspaces $Y\subset D^2 \times [-1,1]$ such that $\varphi_\pi (Y) = Y$ setwise,
    we identify the mapping torus $S^1 \times _{\varphi_\pi} Y$ and the half-spin $H_Y \subset S^1 \times D^2 \times [-1,1]$ of $Y$,
    via the homeomorphism
    \[S^1 \times_{\varphi_\pi} Y \to H_Y : (\theta , y ) \mapsto (\theta , \varphi_{\pi \theta} (y)).\]
\end{rem}

Henceforth, we may add subscript to copies of $S^1$ to distinguish different instances from each other. 
Consider the map $r\beta\times r:A\times S_{v}^{1}\to A\times S_{v}^{1}$.
Let $$X:=S^1 \times_{r\beta \times r} (A\times S_{v}^{1})$$ be its mapping torus.
We illustrate $X$ along with several relevant curves in the boundary in Figure \ref{fig:X}.

The boundary of the mapping torus $X$ is a disjoint union of the \emph{external} boundary $\partialext  X$ (which is the mapping torus of $\partial D^2\times S^1_v$) and the \emph{internal} boundary $\partialint X$ (which is the mapping torus of $(\partial A\setminus\partial D)\times S^1_v$).

Since $\beta$ acts trivially on $\partial D^2$, the external boundary of $X$ is
\begin{equation*}\label{eq:ext-bdry}
\partialext X=S^1\times_{r\times r}(\partial D^2\times S^1_v).
\end{equation*}
The internal boundary of $X$ is also a single component, due to the fact that $r\beta$ acts transitively on the components of $\partial A\setminus\partial D$ (Remark \ref{rem:transitive}). More specifically, we have $$\partialint X=S^1\times_{r\beta\times r}((\partial D_1\sqcup\partial D_2\sqcup\partial D_3\sqcup\partial D_4)\times S^1_v)\cong S^1\times _{(r\times r)^4}(S^1\times S^1)=S^1\times S^1\times S^1,$$
since $r$ is an involution.

\definecolor{realgreen}{rgb}{0,.5,0}
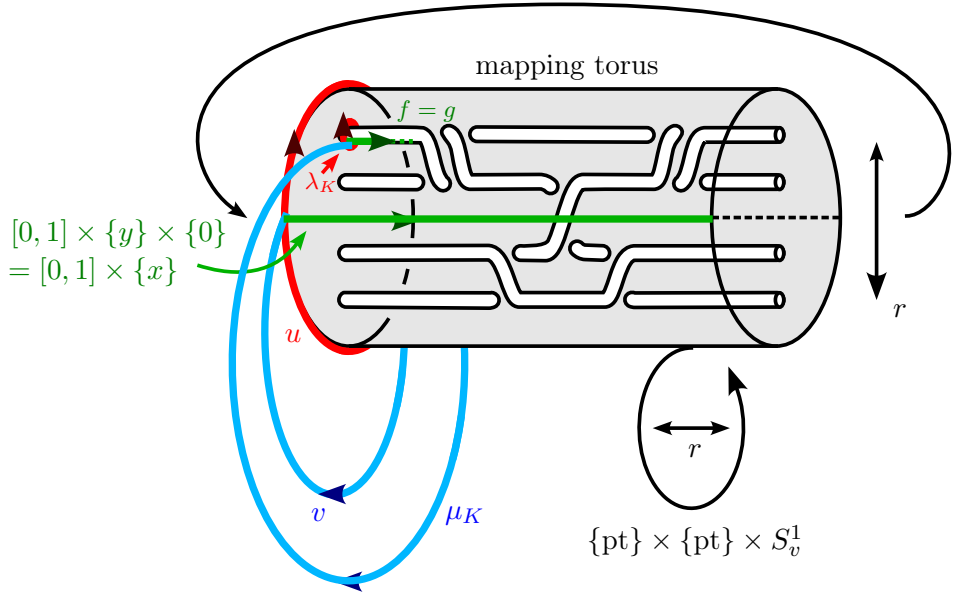
\begin{figure}
\labellist
\pinlabel{$r$} at 275 110
\pinlabel{mapping torus} at 145 205
\pinlabel{$\{\text{pt}\}\times\{\text{pt}\}\times S^1_v$} at 195 20
\pinlabel{{\textcolor{red}{\footnotesize{$\lambda_K$}}}} at 49 160
\pinlabel{\textcolor{blue}{$v$}} at 48 30
\pinlabel{\textcolor{blue}{$\mu_K$}} at 105 30
\pinlabel{\textcolor{red}{$u$}} at 38 100
\pinlabel{\textcolor{realgreen}{$[0,1]\times \{y\}\times \{0\}$}} at -30 140
\pinlabel{\textcolor{realgreen}{$=[0,1]\times\{x\}$}} at -40 125
\pinlabel{\footnotesize{\textcolor{realgreen}{$f=g$}}} at 90 188
\pinlabel{$r$} at 195 55
\endlabellist

\includegraphics[width=100mm]{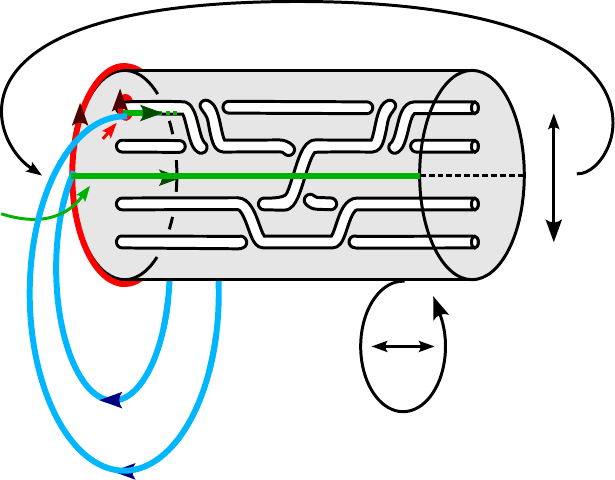}
\caption{The 4-manifold $X=S^1\times_{r\beta\times r}(A\times S^1_v)$. We indicate curves on $\partialint X$ and how they are identified with the boundary of $S^1 \times (S^3\setminus\mathring{N}(K))$ (for $K=T_{3,4}$) in order to form $W$, and also how curves on $\partialext X$ are identified with $S^4\setminus\mathring{N}(Q)$.}\label{fig:X}
\end{figure}

\subsection{\label{subsec:fill-int}Filling the interior boundary of $X$ to form the manifold $W$}

Let $y_i \in \partial D_i ^2$ be the points depicted in Figure~\ref{fig:a-disk}. 
We fix coordinates on $\partialint X=S_{f}^{1}\times S_{a}^{1}\times S_{v}^{1}$, where 
\begin{eqenumerate}
\eqitem\label{enu:intf}$S_{f}^{1}\times \{0\}\times\{0\}$ 
is identified with the loop on $\partialint X$ given by concatenating $[0,1]\times\{y_{1}\}\times \{0\}$, 
$[0,1]\times\{y_{4}\}\times\{0\}$, $[0,1]\times\{y_{3}\}\times\{0\}$,
and $[0,1]\times\{y_{2}\}\times\{0\}$,
\eqitem\label{enu:inta}$\{0\}\times S_{a}^{1}\times\{0\}$
is identified with $\{0\}\times\partial D_1 ^2 \times\{0\}\subset \partialint X$, and
\eqitem\label{enu:intv}$\{0\}\times\{0\}\times S_{v}^{1}$
is identified with $\{0\}\times\{y_{1}\}\times S_{v}^{1}\subset \partialint X$.
\end{eqenumerate}
In particular, the point $(0,y_1 , 0) \in \partialint X$ is identified with the point $(0,0,0)\in S_{f}^{1}\times S_{a}^{1}\times S_{v}^{1}$.

Let $K\subset S^{3}$ be the $(3,4)$ torus knot $T_{3,4}$. 
The boundary of $S^3 \setminus \mathring{N}(K)$ is the torus $S^1 _{\lambda_K} \times S^1 _{\mu _K}$, where $S^1_\lambda\times \{0\}$ is a 0-framed longitude of $K$ and $\{0\}\times\mu_K$ is a meridian of $K$.

Let $W=X\cup (S_{g}^{1}\times(S^{3}\setminus\mathring{N}(K)))$, where we glue $\partial (S^1 _g \times (S^3 \setminus \mathring{N}(K))) = S_{g}^{1}\times S_{\lambda_{K}}^{1}\times S_{\mu_{K}}^{1}$ to $\partial_\mathrm{int} X = S_{f}^{1}\times S_{a}^{1}\times S_{v}^{1}$ 
by identifying the following oriented curves.
\stepcounter{equation}
\begin{align*}
    S^1_f\times\{0\}\times\{0\}&=S^1_g\times\{0\}\times\{0\},\\\{0\}\times S^1_a\times\{0\}
    &=\{0\}\times S^1_{\lambda_{K}}\times\{0\},\tag{\theequation}\label{eq:identify-t3}
    \\
    \{0\}\times\{0\}\times S^1_v&=\{0\}\times\{0\}\times S^1_{\mu_K}.
\end{align*}
Similarly the point $(0,y_1 , 0) \in \partialint X$ is identified with the point $(0,0,0)\in S_{g}^{1}\times S_{\lambda_{K}}^{1}\times S_{\mu_{K}}^{1}$.

Since we are gluing along a 3-torus, whose (not necessarily orientation-preserving) mapping class group is $\GL_3(\mathbb{Z})$ in correspondence with the induced action on first homology, the gluing is determined up to isotopy by these curve identifications.

\begin{rem}[Choice of the knot $K=T_{3,4}$]
The knot $K=T_{3,4}$ is chosen so that the map
$C_2\{\mu_K\}\times C_2\{\lambda_K\}\to\pi_{1}(S^{3}\setminus K)/\left\langle \left\langle \mu^{2}_K,\lambda^{2}_K\right\rangle \right\rangle$ is injective (Lemma~\ref{lem:t34-injective}), where $\mu_K , \lambda_K$ are the meridian and longitude of $K$.
(In the proof of Theorem~\ref{thm:kinoshita}, this is used to show that the map (\ref{eq:use-t34}) is injective.)

\end{rem}

\subsection{Gluing $W$ to $S^4\setminus\mathring{N}(Q)$\label{subsec:gluew}}
We now glue $W$ to the exterior of $S^4\setminus\mathring{N}(Q)$ to form a new closed 4-manifold containing the image of $P$, which we denote as $R$ inside $(S^4\setminus\mathring{N}(Q))\cup W$.

To glue $W$ and $S^4 \setminus \mathring{N} (Q)$, let us specify an identification between $\partialext X=\partial W$ and $\partial N(Q)$.
Consider the solid torus $n_q \subset D^2 \times [-1,1]$ of Figure~\ref{fig:pq-link};
we take $N(Q)$ as the half-spin of $n_q$.
In Figure~\ref{fig:pq-link}~(b), we label two curves in $\partial n_q$.
The curve $u$ is a meridian of $q$, and the curve $v$ is a 0-framed longitude of $q$.
Let us put $D^2\times S^1$ coordinates on $n_q$ where $u$ is $\partial D^2\times\{0\}$ and $v$ is $\{0\}\times S^1$.
The rotation $\varphi_\pi$ restricts to $n_q$ as the map $r\times r: D^2\times S^1\to D^2\times S^1$.
Since $N(Q)$ is the half-spin of $n_q$, we have (recall Remark~\ref{rem:half-spin-mapping-torus})
\[
N(Q) \cong S^1 \times _{\varphi_\pi } n_q = S^1 \times_{r\times r} (D^2 \times S^1).
\]
Now, our identification $\partial W \cong \partial N(Q)$ is given by 
\[
\partial W = \partialext X = S^1\times_{r\times r}(\partial D^2\times S^1_v) \cong  S^1 \times _{r\times r} (\partial D^2 \times S^1) = \partial N(Q),
\]
where the $\cong$ is given by the coordinatewise identification $\partial D^2 \times S^1 _v \cong \partial D^2  \times S^1 $.

For clarity in later computation, we will identify some oriented curves in $\partial W$ and $\partial(S^4\setminus\mathring{N}(Q))$ that are identified by this gluing map.

Let $x\in \partial n_q \subset D^{2}\times[-1,1]$ be the point indicated in Figure~\ref{fig:pq-link}~(b).
Note that $\varphi_{\pi}(x)=x$.
Let $\overline{x}$ denote the corresponding point $(0,x)\in \partial N(Q)$, which is contained in $S^{1}\times D^{2}\times[-1,1]\subset S^{4}$. Identify the $D^{2}\times[-1,1]$ of Figure~\ref{fig:pq-link}~(b) with $\{0\}\times D^{2}\times[-1,1]$.
We will call the loops $u,v$ in $\partial n_q$ by the same names even when viewed as curves in $\partial(S^4\setminus\mathring{N}(Q))$. 
Let $y\in \partial D^2$ be the point depicted in Figure~\ref{fig:a-disk}. 

The given identification between $\partialext X$ and $\partial N(Q)$, as illustrated in Figure \ref{fig:X}, identifies basepoints $(0,y,0)$ in $\partial W$ with $\overline{x}$ in $\partial(S^4\setminus\mathring{N}(Q))$. The gluing also identifies the following pairs of oriented curves in $\partial W$ and $\partial(S^4\setminus\mathring{N}(Q))$.
\stepcounter{equation}
\begin{align*}
\{0\}\times \partial D^2 \times \{0\}&=u,\\
\{0\}\times \{y\} \times S_v^1&=v,\tag{\theequation}\label{eq:extgluing}
\\
[0,1]\times \{y\} \times \{0\}&=[0,1]\times \{x\}.
\end{align*}
Note both $[0,1]\times\{y\} \times \{ 0\} \subset\partial W$ and $[0,1]\times\{x\}\subset\partial(S^4\setminus\mathring{N}(Q))$ are closed loops.

\subsection{Understanding the new ambient manifold (is $S^4$)}\label{sec:ambientstills4}

We must now understand the topology of the 4-manifold $(S^{4}\setminus\mathring{N}(Q))\cup W$ that contains the projective plane $R$. While we took $X$ and $W$ to be unoriented (but orientable), the union $(S^{4}\setminus\mathring{N}(Q))\cup W$ inherits an orientation from $S^4\setminus\mathring{N}(Q)$.
We claim this manifold is diffeomorphic to $S^4$.

\begin{lem}[Generalized knot surgery gives back $S^4$]
\label{lem:get-s4-back}
The union $((S^1 \times D^2 \times [-1,1])\setminus \mathring{N}(Q))\cup W$ and $S^1 \times D^2 \times [-1,1]$ are diffeomorphic.
Hence, $(S^{4}\setminus\mathring{N}(Q))\cup W$ is diffeomorphic to $S^4$.
\end{lem}

\begin{proof}
Let us first give an alternative description of $W$. Recall $K$ is the torus knot $T_{3,4}$. Let $Y$ be
the $3$-manifold given by gluing four copies of $S^{3}\setminus\mathring{N}(K)$ to $A\times S_{v}^{1}$ as follows. 
Denote the four copies as $(S^{3}\setminus\mathring{N}(K))_{i}$ for $i=1,2,3,4$, and let $\iota_i :\partial(S^{3}\setminus\mathring{N}(K))_i\to\partial D_{i}^{2}\times S_{v}^{1}$
be a diffeomorphism with $\lambda_{K}\mapsto\partial D_{i}^{2}\times\{0\}$
and $\mu_{K}\mapsto\{y_i \}\times S_{v}^{1}$. (Note this determines a diffeomorphism up to isotopy.) 
Let $r_i :D_i ^2 \to D_i ^2 $ be the horizontal reflection. 
Glue $(S^{3}\setminus\mathring{N}(K))_{i}$
to $\partial D_{i}^{2}\times S_{v}^{1}$; we use the gluing map $\iota_i$
for $i=1,3$, and the gluing map $(r_i\times r)\circ\iota_i $ for $i=2,4$.

Let $r\beta:Y\to Y$ be the diffeomorphism given by $r\beta\times r$ on
$A\times S_{v}^{1}$ and
the map $\mathrm{Id}:(S^{3}\setminus\mathring{N}(K))_{i}\to(S^{3}\setminus\mathring{N}(K))_{r\beta (i)}$ on  $(S^{3}\setminus\mathring{N}(K))_{i}$ ($i=1,2,3,4$).
Then $W$ is diffeomorphic
to the mapping torus $S^1 \times _{r\beta\times r} Y$, under which $X\subset W$
is identified with $S^1 \times _{r\beta\times r} (A\times S_v^1)$, as usual.

Consider the solid torus $n_q\subset D^{2}\times[-1,1]$ of Figure~\ref{fig:pq-link}~(a).
Recall that $\varphi_{\pi}(n_q)=n_q$ and that its half-spin is $N(Q)$.
In Subsection~\ref{subsec:gluew}, we made the identification $n_q\cong D^2 \times S^1$.
Let $B:=(D^2 \times [-1,1]\setminus\mathring{n_q})\cup Y$,
where we glue them along $\partial n_q \cong \partial D^2 \times S^1 $ and $\partial Y = \partial D^2 \times S^1 $ by the identity.
Let $\eta:B\to B$
be the diffeomorphism given by $\varphi_{\pi}$ on $B\setminus\mathring{n_q}$
and $r\beta\times r$ on $Y$.
Note that attaching a $3$-dimensional $2$-handle to $S^3 \setminus \mathring{N} (K)$ along the meridian $\mu_K$ results in a $3$-ball.
Using this observation $4$ times, we have that $B$ is a $3$-ball.
Hence $\eta$ is isotopic
to $\mathrm{Id}_B$ by Cerf--Palais \cite{MR140120, MR117741},
and so the following are diffeomorphic:
\[((S^1 \times D^2 \times [-1,1])\setminus \mathring{N}(Q))\cup W \cong S^1 \times _\eta B\cong S^1 \times_{\mathrm{Id}_B} B \cong S^1 \times D^2 \times [-1,1].\]

For the last sentence of Lemma~\ref{lem:get-s4-back}, note that the complement of $S^1 \times D^2 \times [-1,1]\cong S^1\times B^3$ in $S^4$ is $D^2 \times S^2$.
Thus we have
\begin{align*}
(S^{4}\setminus\mathring{N}(Q))\cup W &
\cong \left ( ((S^1 \times D^2 \times [-1,1])\setminus \mathring{N}(Q))\cup W \right) \cup (D^2 \times S^2 ) \\
&\cong (S^1 \times B^3 ) \cup (D^2 \times S^2 ).
\end{align*}
The last sentence follows since $(S^1 \times B^3 ) \cup (D^2 \times S^2 )$ is diffeomorphic to $S^4$ (regardless of the diffeomorphism of $S^1 \times S^2$ we use to glue $S^1 \times B^3 $ and $D^2 \times S^2 $).
\end{proof}

The following is our main theorem, which will be proved in Subsection~\ref{subsec:irreducible-p2}.

\begin{thm}[An irreducible projective plane in $S^4$]\label{thm:main-thm}
The image of $P$ in $(S^{4}\setminus\mathring{N}(Q))\cup W\cong S^{4}$
is an irreducible projective plane in $S^4$.
\end{thm}

\begin{rem}[Our example is a deform-spun projective plane]\label{rem:spinning}
The proof of Lemma~\ref{lem:get-s4-back} shows that our irreducible projective plane from Theorem~\ref{thm:main-thm} is a deform-spin (\cite[Section III]{PriceRoseman1975},  \cite[Section 6.4]{MR3588325}; c.f.\ \cite{litherland1979deforming}) of $\#_4 T_{3,4}$,
the connected sum of four copies of the torus knot $T_{3,4}$. The deformation encodes the braid $\beta$, and notably is {\emph{not}} a combination of rolls and twists (composed with a strong involution), as remarked in the history of spinning included in Section~\ref{sec:intro}.

\begin{rem}[Alternative description]
Alternatively, our irreducible projective plane can be obtained from $P$ by first performing a satellite operation to the Klein bottle $Q$ (yielding a torus $T$) and then performing knot surgery on $T$.
In Figure \ref{fig:linkoftorusrp2},  we present a banded unlink diagram (\cite{kssnormalform}, c.f.\ \cite{bandedunlink}) of $P\sqcup T$.
\end{rem}
\begin{figure}
\labellist
\pinlabel{\huge{$\frac{1}{2}$}} at 159 162
\pinlabel{\Large{$P$}} at 55 15
\pinlabel{\Large{$T$}} at 17 179
\endlabellist
    \includegraphics[width=110mm]{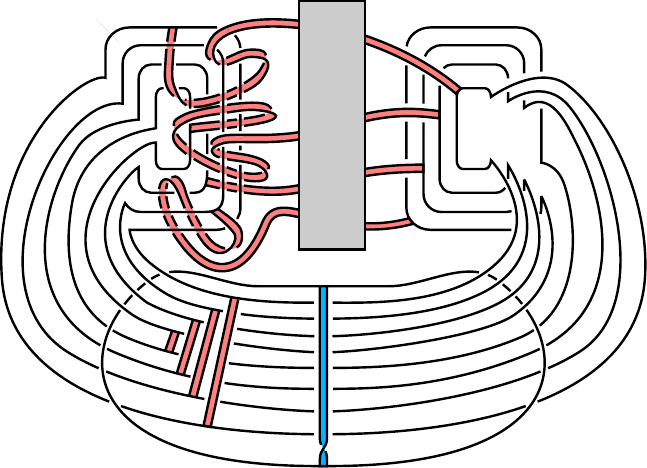}
    \caption{An unknotted projective plane $P$ (here with negative normal Euler number) and a torus $T$. Performing knot surgery along $T$ in $(S^4, P)$ yields $(S^4, R)$. The twisting of the bands in the upper left of the diagram represents the braid $\beta$ (this is the effect of applying $\beta$ to the four unlink strands with bands leaving straight to the right).}\label{fig:linkoftorusrp2}
\end{figure}
\end{rem}

\section{\label{sec:proof}The proofs}
\subsection{\label{subsec:irreducible-p2}An irreducible projective plane in $S^4$}
Let $R$ be the image of $P$ in $(S^{4}\setminus\mathring{N}(Q))\cup W\cong S^{4}$.
In this subsection we show that $R$ is irreducible and hence prove Theorem~\ref{thm:kinoshita}.
To show that $R$ is irreducible, we use Lemma~\ref{lem:peripheral-obstruction}
(this appears in e.g.\ \cite[Remark~3.7]{katanaga1998embeddings}).

\begin{lem}
\label{lem:peripheral-obstruction}
    If a projective plane $R\subset S^4$ is reducible, then the image of the peripheral map $\pi_1 (\partial N(R))\to \pi _1 (S^4 \setminus R)$ has order two.
\end{lem}
\begin{proof}
    Suppose $R$ is a connected sum of a knotted 2-sphere $K$ and an unknotted projective plane $P$, so $(S^4,R)\cong-(B^4_1,K^\circ)\cup(B^4_2,P^\circ)$, where the superscript `$^\circ$' denotes a tangle obtained by deleting a trivial disk. Take the basepoint $z$ of $\pi_1(S^4\setminus N(R))$ to be on the boundary and in $B^4_2$, so that $z$ is also a basepoint for $\pi_1(B^4_2\setminus N(P^\circ))$. Note $\pi_1(\partial N(R)\cap B^4_2)$ surjects via inclusion to $\pi_1(\partial N(R))$, i.e.\ the boundary fundamental group is generated by curves in $B^4_2$. Since $\pi_1(B^4_2\setminus P^\circ )\cong C_2$, we conclude that the image of $\pi_1 (\partial N(R))\to \pi _1 (S^4 \setminus R)$ has order at most two --- and it must be equal to two, since the meridian maps to a nontrivial element.
\end{proof}

Recall from Subsection~\ref{subsec:gluew} that we identify the $D^{2}\times[-1,1]$ of Figure~\ref{fig:pq-link} with $\{0\}\times D^{2}\times[-1,1]\subset S^{1}\times D^{2}\times[-1,1]\subset S^{4}$.
The basepoint for our fundamental group computations will be $\overline{x}:=(0,x)\in S^{1}\times D^{2}\times[-1,1]\subset S^{4}$.
Also recall that we view the loops $u,v$ from Figure~\ref{fig:pq-link}~(b) as loops in $\partial N(Q)$ (and hence in $S^{4}\setminus\mathring{N}(P\sqcup Q)$) based at $\overline{x}$.

\begin{lem}
\label{lem:pi1-maps-pq}Let $w:= [0,1]\times \{x\}\subset \partial N(Q)$. The map $\pi_{1}(\partial N(Q),\overline{x})\to\pi_{1}(S^{4}\setminus(P\sqcup Q),\overline{x})$ is
\[\left\langle u,v,w\middle|[u,v],wuw^{-1}u,wvw^{-1}v\right\rangle \to C_2 \{u\} \times C_2 \{ v\}: u\mapsto u,\ v\mapsto v ,\ w\mapsto e,\] where $e$ denotes the identity element.
\end{lem}

\begin{proof}
Consider the arc $p$ and circle $q$ in $D^2 \times [-1,1]$ of Figure~\ref{fig:pq-link}~(a). 
Then, 
\[
\pi_{1}((D^{2}\times[-1,1])\setminus(p\sqcup q),x)= C_{\infty} \{u\} \times C_\infty \{ v\},
\]
and $\varphi_{\pi}(u)=u^{-1}$, $\varphi_{\pi}(v)=v^{-1}$.

Let $z:=(\mathbf{0},1)\in D^{2}\times[-1,1]$; note that $\varphi_{\pi}(z)=z$.
Consider the mapping torus of $\varphi_{\pi}$ on $(D^{2}\times[-1,1])\setminus\mathring{N}(p\sqcup q)$
and attach a $2$-cell along $[0,1]\times \{z\}$; this is homotopy equivalent
to $S^{4}\setminus(P\sqcup Q)$. Since there is a path from $x$ to $z$ in $D^2 \times [-1,1] \setminus (p\sqcup q)$ that is invariant under $\varphi_\pi$, we find that $[0,1]\times \{ x\} = w$ is nullhomotopic in $S^4 \setminus (P\sqcup Q)$. Hence, 
\begin{align*}
\pi_{1}(S^{4}\setminus(P\sqcup Q),\overline{x}) & \cong\bigg(\pi_{1}((D^{2}\times[-1,1])\setminus(p\sqcup q),x)\bigg)/\left\langle \left\langle \xi^{-1}\varphi_{\pi}(\xi)\ \forall \xi \right\rangle \right\rangle \\
 & =C_2 \{u\} \times C_2 \{ v\}.\qedhere
\end{align*}
\end{proof}

The following is the key fundamental group computation.

\begin{thm}[Main computation]
\label{thm:main-pi1-comp}The map 
\[
C_2 \{u\} \times C_2 \{ v\} \to 
\pi_{1}\left(\left(S^{4}\setminus(P\sqcup\mathring{N}(Q))\right)\cup W,\overline{x}\right)
\]
is injective.
\end{thm}

\begin{proof}[Proof of Theorems~\ref{thm:kinoshita}~and~\ref{thm:main-thm} assuming Theorem~\ref{thm:main-pi1-comp}]
We will show that Theorem~\ref{thm:main-pi1-comp} implies that the image of the peripheral map $\pi_1 (\partial N(R))\to \pi _1 (S^4 \setminus R)$ has order at least $4$.
Once this is established, Theorem~\ref{thm:main-thm}, and hence Theorem~\ref{thm:kinoshita}, follows from Lemma~\ref{lem:peripheral-obstruction}.

We need the basepoint to be on $\partial N(R)$. For ease of visualization, let us change the basepoint from $\overline{x}$ to $\overline{x'}$:
consider the $x',u',v'$ of Figure~\ref{fig:pq-link}~(c), let $\overline{x'}:=(0,x')\in S^{1}\times D^{2}\times[-1,1]\subset S^{4}$, and let $u',v'$ be the corresponding loops in $\{0\}\times D^{2}\times[-1,1]\subset S^{1}\times D^{2}\times[-1,1]\subset S^{4}$ based at $\overline{x'}$.
By Theorem~\ref{thm:main-pi1-comp},
\begin{equation}\label{eq:pi1-inj-change-bastpt}
C_2 \{u'\} \times C_2 \{ v'\} \to 
\pi_{1}\left(\left(S^{4}\setminus(P\sqcup\mathring{N}(Q))\right)\cup W,\overline{x'}\right)
\end{equation}
is injective.

Let $N(P)$ be the union of the southern hemisphere $B_S ^4$ and the half-spin of the neighborhood $n_p \subset D^2 \times [-1,1]$ of $p$ depicted in Figure~\ref{fig:pq-link}~(c).
Then, $N(P)$ is a tubular neighborhood of $P$ (and is disjoint from $N(Q)$),
and the basepoint $\overline{x'}$ and the loops $u', v'$ lie on $\partial N(P)$.
Hence the image of  
\begin{equation}\label{eq:pi1-bdry-ext}
\pi_1 (\partial N(P),\overline{x'})\to \pi_{1}\left(\left(S^{4}\setminus(P\sqcup\mathring{N}(Q))\right)\cup W,\overline{x'}\right)
\end{equation}
contains the image of the map~(\ref{eq:pi1-inj-change-bastpt}), and hence has order at least $4$.
In fact, the image of the map~(\ref{eq:pi1-bdry-ext}) has order $4$ since ${v'}^2$ is in the kernel;
the kernel is the center $\{\pm1\} \subset Q_8 \cong \pi_1 (\partial N(P))$.
\end{proof}

Let us record some additional lemmas before we prove Theorem~\ref{thm:main-pi1-comp}.

\begin{lem}
\label{lem:t34-injective}Let $\mu_K, \lambda_K \in \pi_{1}(S^{3}\setminus K)$ be the meridian and longitude of $K=T_{3,4}$.
Then, the map $C_2\{\mu_K\}\times C_2\{\lambda_K\}\to\pi_{1}(S^{3}\setminus K)/\left\langle \left\langle \mu^{2}_K,\lambda^{2}_K\right\rangle \right\rangle$ is injective.
\end{lem}

\begin{figure}
    \centering
    \includegraphics[width=0.85\linewidth]{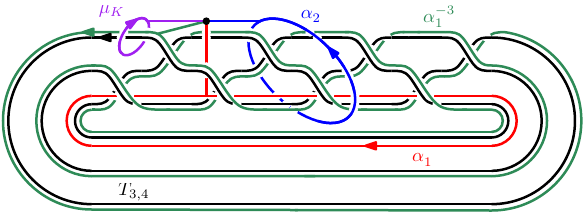}
    \caption{The torus knot $T_{3,4}$ (black), with the meridian $\mu_K$ and curves $\alpha_1$, $\alpha_2$ indicated.
    The green curve is homotopic to $\alpha_1^{-3}$, and is a parallel to the knot $T_{3,4}$ with surface framing $12$ induced by the torus.  We  obtain $\lambda_K$ from $\alpha_1^{-3}$ by adding 12 twists in the direction of $\mu_K^{-1}$.}
    \label{fig:torusknot}
\end{figure}

\begin{proof}
Using the usual description of the torus knot $K=T_{3,4}$ as a subset of the Heegaard surface of the genus $1$ Heegaard splitting of $S^3$, one can show $\pi_{1}(S^{3}\setminus K)\cong\left\langle \alpha_1 ,\alpha_2 \middle|\alpha_1^{3}\alpha_2^{4}\right\rangle $ by Seifert--van Kampen and that $\mu_K=\alpha_1^{-1}\alpha_2^{-1}$ and $\lambda_K=(\alpha_2 \alpha_1 )^{12}\alpha_1^{-3}$ (see Figure~\ref{fig:torusknot}).
Now, this lemma can be checked using the following GAP code \cite{GAP4}.
\begin{verbatim}
F := FreeGroup("a1", "a2");
a1 := F.1; a2 := F.2;
G := F / [a1^3 * a2^4 , (a1^-1 * a2^-1)^2 , ((a2 * a1)^12 * a1^-3)^2];
a1_g := G.1; a2_g := G.2;
mu := a1_g ^-1 * a2_g^-1; lambda := (a2_g * a1_g)^12 * a1_g^-3;
Print("Order(G)=", Order(G), ", Order(mu)=", Order(mu),
", Order(lambda)=", Order(lambda), ", Order(mu*lambda)=", Order(mu * lambda));
\end{verbatim}
Indeed, GAP outputs
\begin{verbatim}
Order(G)=48, Order(lambda)=2, Order(mu)=2, Order(lambda*mu)=2
\end{verbatim}
\vspace{-1.7\baselineskip}
\qedhere
\end{proof}

\begin{lem}
\label{lem:alg-lemma1}Let $H,G_{1},G_{2}$ be groups, let $f_{i}:H\to G_{i}$
be a homomorphism for $i=1,2$, and let $S$ be a subset of $H$.
If $f_{1}(s)\in G_1$ is the identity for all $s\in S$, then the map $\mathrm{Id}_{G_1} :G_1 \to G_1$ and the quotient map $\pi : G_2 \to G_2 / \left\langle \left\langle f_{2}(S)\right\rangle \right\rangle$ induce an isomorphism
\[
\mathrm{Id}_{G_1} \ast \pi : G_{1}\ast_{H}G_{2}\xrightarrow{\cong} G_{1}\ast_{H/\left\langle \left\langle S\right\rangle \right\rangle }G_{2}/\left\langle \left\langle f_{2}(S)\right\rangle \right\rangle .
\]
\end{lem}

\begin{proof}
Let $j_i : G_{i} \to G_1 \ast_H G_2$ be the pushout maps.
Since $j_1 \circ f_1 =j_2 \circ f_2$ and $f_1 (s)$ is the identity for all $s\in S$, we have that $j_2 (f_2 (s))$ is the identity for all $s\in S$.
Hence, $j_2$ factors through $G_2 /  \left\langle \left\langle f_2 (S)\right\rangle \right\rangle$; denote the induced map as $j_2':G_2 /  \left\langle \left\langle f_2 (S)\right\rangle \right\rangle \to G_1 \ast _H G_2$.
The maps $j_1$ and $j_2'$ induce a map $ G_{1}\ast_{H/\left\langle \left\langle S\right\rangle \right\rangle }G_{2}/\left\langle \left\langle f_{2}(S)\right\rangle \right\rangle \to G_{1}\ast_{H}G_{2} $ which is the inverse of $\mathrm{Id}_{G_1} \ast \pi$.
\end{proof}

\begin{lem}[{\cite[Remark after Theorem~1, Section~I.2]{MR607504}}]
\label{lem:alg-thm2}Let $H,G_{1},G_{2}$ be groups, and let $f_{i}:H\to G_{i}$
be injective homomorphisms for both $i=1,2$. Then the
pushout maps $G_{i}\to G_{1}\ast_{H}G_{2}$ are injective.
\end{lem}

Now we are ready to prove Theorem~\ref{thm:main-pi1-comp}.
\begin{proof}[Proof of Theorem~\ref{thm:main-pi1-comp}]
For ease of exposition, we divide the proof into four steps.

\vspace{3mm}
\noindent\textbf{Step 1} (Reduce to showing that  $C_2\{u\}\times C_2\{v\}\to\pi_{1}(W,\overline{x})/\left\langle \left\langle [0,1]\times\{x\}\right\rangle \right\rangle$ is injective)\textbf{.}
By Lemma~\ref{lem:pi1-maps-pq},
the map $\pi_{1}(\partial N(Q),\overline{x})\to \pi_{1}(S^{4}\setminus(P\sqcup \mathring{N}(Q)),\overline{x})$
is the quotient map 
\[
\pi_{1}(\partial N(Q),\overline{x})
\to \pi_{1}(\partial N(Q),\overline{x})/\left\langle \left\langle [0,1]\times \{x\} \right\rangle\right\rangle.
\]
Hence by Seifert--van Kampen we have
\begin{align*}
\pi_{1}\left(\left(S^{4}\setminus(P\sqcup\mathring{N}(Q))\right)\cup W,\overline{x}\right) &
\cong \pi_{1}(S^{4}\setminus(P\sqcup \mathring{N}(Q)),\overline{x})\ast_{\pi_{1}(\partial N(Q),\overline{x})} \pi_{1}(W,\overline{x})\\
& \cong  \pi_{1}(W,\overline{x})/\left\langle \left\langle [0,1]\times\{x\}\right\rangle \right\rangle .
\end{align*}
Thus Theorem~\ref{thm:main-pi1-comp} is equivalent to the map
\begin{equation}\label{eq:want-inj}
    C_2\{u\}\times C_2\{v\}\to\pi_{1}(W,\overline{x})/\left\langle \left\langle [0,1]\times\{x\}\right\rangle \right\rangle
\end{equation}
being injective. Note that since $u,v \subset \partial N(Q) = \partial_\mathrm{ext} X \subset X$, the map (\ref{eq:want-inj}) factors through
\begin{equation}\label{eq:map-to-x}
C_2\{u\}\times C_2\{v\}\to \pi_{1}(X,\overline{x})/\left\langle \left\langle [0,1]\times\{x\}\right\rangle \right\rangle.
\end{equation}

\vspace{3mm}
\noindent\textbf{Step 2} (Express $\pi_1 (W,\overline{x})/\left\langle \left\langle [0,1]\times\{x\}\right\rangle \right\rangle$ using Seifert--van Kampen)\textbf{.}
In this step we use Seifert--van Kampen 
to construct the following isomorphism (we explain notation below).
\begin{multline}\label{eq:svk-W}
\pi_1 (W,\overline{x})/\left\langle \left\langle [0,1]\times\{x\}\right\rangle \right\rangle \\
\cong (\pi_{1}(X,\overline{x}) /\left\langle \left\langle [0,1]\times\{x\}\right\rangle \right\rangle)\ast_{ C_{\infty}\{f\}\times C_{\infty}\{a\}\times C_{\infty}\{v\}}\pi_{1}(S_{g}^{1}\times(S^{3}\setminus\mathring{N}(K)),(0,0,0)).
\end{multline}

Let $W' := W\cup e^2$, $X' := X\cup e^2 \subset W'$ where $e^2$ is a $2$-cell glued along $[0,1]\times \{x\}$.
We identify
\[\pi_1 (W,\overline{x})/\left\langle \left\langle [0,1]\times\{x\}\right\rangle \right\rangle \cong \pi_1 (W',\overline{x}),
\ \pi_1 (X,\overline{x})/\left\langle \left\langle [0,1]\times\{x\}\right\rangle \right\rangle \cong \pi_1 (X',\overline{x}).
\]
The isomorphism (\ref{eq:svk-W}) is given by Seifert--van Kampen on the two subsets
\begin{equation}\label{eq:two-subsets}
X'\quad \mathrm{and} \quad (S_{g}^{1}\times(S^{3}\setminus\mathring{N}(K)))\cup(\{0\}\times \rho_{1}\times\{0\}) \subset W'.
\end{equation}
Here, we added $\{0\}\times \rho_{1}\times\{0\}\subset\{0\}\times A\times\{0\}\subset X$
so that the second subset contains the basepoint $\overline{x}$.
(Recall $X$ is the mapping
torus of $r\beta\times r:A\times S_{v}^{1}\to A\times S_{v}^{1}$
and $\overline{x}=(0,y,0)$ and $\rho_1$ is the arc in $A$ pictured in Figure \ref{fig:a-disk}.)

The intersection of the two subsets (\ref{eq:two-subsets}) is $\partialint X \cup (\{0\}\times \rho_{1}\times\{0\})$.
In Subsection~\ref{subsec:fill-int} we made the identification $\partialint X = S_{f}^{1}\times S_{a}^{1}\times S_{v}^{1}$.
Hence,
\[
\pi_1 (\partialint X \cup (\{0\}\times \rho_{1}\times\{0\}),\overline{x}) = C_{\infty}\{f\}\times C_{\infty}\{a\}\times C_{\infty}\{v\}
\]
where $f,a,v$ are loops in $\partialint X \cup (\{0\}\times \rho_{1}\times\{0\})$ based at $\overline{x}$ (and so also are loops in $X$), defined as follows.
(Note that the two endpoints of $\{0\}\times \rho_{1}\times\{0\}$ are $\overline{x}=(0,y,0)$ and $(0,y_1,0) \in \partialint X$,
and the identification
$\partialint X = S_{f}^{1}\times S_{a}^{1}\times S_{v}^{1}$ identifies $(0,y_1,0) \in \partialint X$ with $(0,0,0)\in S_{f}^{1}\times S_{a}^{1}\times S_{v}^{1}$.)
\begin{eqenumerate}
    \eqitem\label{enu:f}The loop $f$ is the concatenation of $\{0\}\times \rho_{1}\times\{0\}$, $S_{f}^{1} \times \{0\} \times \{0\}$, and $\{0\}\times \rho_{1}^{-1}\times\{0\}$. By (\ref{enu:intf}), $f$ is homotopic in $X$ to the loop given by concatenating $[0, 1]\times\{y_1\}\times\{0\},
[0, 1] \times \{y_4\} \times \{0\}, [0, 1] \times\{y_3\}\times \{0\}$, and $[0, 1]\times \{y_2\} \times \{0\}.$
     \eqitem\label{enu:a}The loop $a$ is the concatenation of $\{0\}\times \rho_{1}\times\{0\}$, $\{0\}\times S_{a}^{1} \times \{0\} $, and $\{0\}\times \rho_{1}^{-1}\times\{0\}$.
     By (\ref{enu:inta}), 
     the loop $\{0\}\times S_{a}^{1} \times \{0\} \subset S_{f}^{1}\times S_{a}^{1}\times S_{v}^{1}$ is identified with $\{0\}\times \partial D_1^2 \times \{0\}\subset \partialint X$, and so $a$ is homotopic in $X$ to the loop $\{0\}\times \gamma_1 \times \{0\}$.
     \eqitem\label{enu:v}The loop $v$ is the concatenation of $\{0\}\times \rho_{1}\times\{0\}$, $\{0\} \times \{0\} \times S_v^1$, and $\{0\}\times \rho_{1}^{-1}\times\{0\}$.
     By~(\ref{enu:intv}),
     the loop $\{0\} \times \{0\} \times S_v^1 \subset S_{f}^{1}\times S_{a}^{1}\times S_{v}^{1} $ is identified with $\{0\}\times \{y_1 \}  \times S_v^1 \subset  \partialint X$.
     Hence $v$ is homotopic in $X$ to the loop $\{0\}\times \{y\}  \times S_v^1 $
     and thus to the loop $v \subset \partialext X = \partial N(Q)$ from before (recall Equation~(\ref{eq:extgluing})),
     i.e.\ the $v$ of Equations~(\ref{eq:want-inj})~and~(\ref{eq:map-to-x}).
\end{eqenumerate}

In Subsection~\ref{subsec:fill-int} we made the identification $\partial (S_{g}^{1}\times(S^{3}\setminus\mathring{N}(K))) = S_{g}^{1}\times S_{\lambda_{K}}^{1}\times S_{\mu_{K}}^{1}$.
We use the basepoint $(0,0,0)\in S_{g}^{1}\times S_{\lambda_{K}}^{1}\times S_{\mu_{K}}^{1}$
for $\pi_1(S_{g}^{1}\times(S^{3}\setminus\mathring{N}(K)))$ and identify
\begin{equation}\label{eq:pi1-retraction}
\pi_1((S_{g}^{1}\times(S^{3}\setminus\mathring{N}(K)))\cup(\{0\}\times \rho_{1}\times\{0\}), \overline{x}) \cong \pi_1 (S_{g}^{1}\times(S^{3}\setminus\mathring{N}(K)),(0,0,0))
\end{equation}
using the retraction that contracts $\{0\}\times \rho_{1}\times\{0\}$ to the basepoint $(0,0,0)$ (which is identified with $(0,y_1,0)\in \partialint X$).
Define loops $g$, $\lambda_K$, and $\mu_K$ in $S_{g}^{1}\times S_{\lambda_{K}}^{1}\times S_{\mu_{K}}^{1}$ based at $(0,0,0)$
as $S_{g}^1 \times\{0\}\times\{0\}$, $\{0\}\times S_{\lambda_K} ^1 \times \{0\}$, and $\{0\}\times \{0\}\times S_{\mu_K}^1$ (respectively).

Finally, let us explicitly describe the two maps
\begin{gather}
   \label{eq:pushout-first}C_{\infty}\{f\}\times C_{\infty}\{a\}\times C_{\infty}\{v\} \to \pi_{1}(X,\overline{x})/\left\langle \left\langle [0,1]\times\{x\}\right\rangle \right\rangle\\
    \label{eq:pushout-second}C_{\infty}\{f\}\times C_{\infty}\{a\}\times C_{\infty}\{v\} \to \pi_{1}(S_{g}^{1}\times(S^{3}\setminus\mathring{N}(K)), (0,0,0))
\end{gather}
for the pushout of (\ref{eq:svk-W}).
By definition, the generators $f,a,v\in C_{\infty}\{f\}\times C_{\infty}\{a\}\times C_{\infty}\{v\}$
map to the loops $f,a,v$ (respectively) under the map (\ref{eq:pushout-first}).
By Equation~(\ref{eq:identify-t3}), the identification (\ref{eq:pi1-retraction}) maps
the loops $f,a,v$ to $g,\lambda_K , \mu_K$, respectively, and so the map (\ref{eq:pushout-second}) is given by
\begin{equation}\label{eq:svk-map}
f\mapsto g,\ a \mapsto \lambda_K,\ v \mapsto \mu_K .
\end{equation}

\vspace{3mm}
\noindent\textbf{Step 3} (Prove that (\ref{eq:want-inj}) is injective assuming a proposed map $\psi$ exists)\textbf{.}
In Step 4, we will define a map 
\[\psi:\pi_{1}(X,\overline{x})/\left\langle \left\langle [0,1]\times\{x\}\right\rangle \right\rangle \to S_{4}\times C_2\{t\}\]
such that 
\begin{equation}\label{eq:psi-conds}
\psi (f)=(3\ 4),\ \psi(a)=(1\ 2),\ \psi(u)=(1\ 2)(3\ 4),\ \psi (v) = t.
\end{equation}
Here, $S_4$ denotes the symmetric group on four elements.

In this step we show that the map (\ref{eq:want-inj}) is injective assuming that such a map $\psi$ exists.
Let us use the basepoint $(0,0,0)$ for $S_g^1 \times (S^3 \setminus \mathring{N}(K))$ (as in Step~2), and $(0,0)$ for $S^3 \setminus \mathring{N}(K)$.
We suppress these basepoints from the notation.
Consider the following sequence of maps:
\begin{align}
 & \nonumber
 \pi_{1}(W,\overline{x})/\left\langle \left\langle [0,1]\times\{x\}\right\rangle \right\rangle \\
 & \label{eq:4} \cong\left(\pi_{1}(X,\overline{x})/\left\langle \left\langle [0,1]\times\{x\}\right\rangle \right\rangle \right)\ast_{ C_{\infty}\{f\}\times C_{\infty}\{a\}\times C_{\infty}\{v\}}\pi_{1}(S_{g}^{1}\times(S^{3}\setminus\mathring{N}(K)))\\
 & \label{eq:5} \xrightarrow{\psi \ast \mathrm{Id}} (S_{4}\times C_2\{t\})\ast_{ C_{\infty}\{f\}\times C_{\infty}\{a\}\times C_{\infty}\{v\}}\pi_{1}(S_{g}^{1}\times(S^{3}\setminus\mathring{N}(K)))\\
 & \label{eq:6} \cong(S_{4}\times C_2\{t\})\ast_{ C_2\{f\}\times C_2\{a\}\times C_2\{v\}}\pi_{1}(S_{g}^{1}\times(S^{3}\setminus\mathring{N}(K)))/\left\langle \left\langle g^{2},\lambda_{K}^{2},\mu_{K}^{2}\right\rangle \right\rangle\\
 & \cong(S_{4}\times C_2\{t\})\ast_{ C_2\{f\}\times C_2\{a\}\times C_2\{v\}}\left( C_2\{g\}\times \left(\pi_{1}(S^{3}\setminus\mathring{N}(K))/\left\langle \left\langle \lambda_{K}^{2},\mu_{K}^{2}\right\rangle \right\rangle \right) \right).\label{eq:final-group}
\end{align}
The map (\ref{eq:4}) is the isomorphism (\ref{eq:svk-W}) from Step 2.
The map (\ref{eq:5}) is induced by $\psi$; note that this is surjective, but we do not make use of this fact.
The map (\ref{eq:6}) is given by Lemma~\ref{lem:alg-lemma1}:
indeed, the images $\psi (f^2) ,\psi(a^2) ,\psi(v^2) \in S_4 \times  C_2\{t\}$ are trivial by Equation~(\ref{eq:psi-conds}).
Finally, the map (\ref{eq:final-group}) is given by the isomorphism
\[
\pi_{1}(S_{g}^{1}\times(S^{3}\setminus\mathring{N}(K)))/\left\langle \left\langle g^{2},\lambda_{K}^{2},\mu_{K}^{2}\right\rangle \right\rangle \cong  C_2\{g\}\times \left(\pi_{1}(S^{3}\setminus\mathring{N}(K))/\left\langle \left\langle \lambda_{K}^{2},\mu_{K}^{2}\right\rangle \right\rangle \right).
\]

Let $G$ be the target group in (\ref{eq:final-group}), and let $j:S_4 \times C_2 \{t\} \to G$ be the pushout map.
We use that the following diagram~(\ref{eq:com-diag}) commutes to show that the map (\ref{eq:want-inj}) is injective.
\begin{equation}\label{eq:com-diag}
\begin{tikzcd}[ampersand replacement=\&]
	\& {\pi_{1}(X,\overline{x})/\left\langle \left\langle [0,1]\times\{x\}\right\rangle \right\rangle} \& {\pi_{1}(W,\overline{x})/\left\langle \left\langle [0,1]\times\{x\}\right\rangle \right\rangle} \\
	{ C_2\{u\}\times C_2\{v\}} \& {S_{4}\times C_2\{t\}} \& G
	\arrow[from=1-2, to=1-3]
	\arrow["\psi"{description}, from=1-2, to=2-2]
	\arrow["{\mathrm{(\ref{eq:4})-(\ref{eq:final-group})}}"{description}, from=1-3, to=2-3]
	\arrow["{\mathrm{(\ref{eq:map-to-x})}}"{description}, from=2-1, to=1-2]
	\arrow["{(\ref{eq:want-inj})}"{description, pos=0.7}, dashed, from=2-1, to=1-3]
	\arrow["{\psi \circ \mathrm{(\ref{eq:map-to-x})}}"{description}, from=2-1, to=2-2]
	\arrow["j"{description}, from=2-2, to=2-3]
\end{tikzcd}
\end{equation}
We claim that $\psi \circ \mathrm{(\ref{eq:map-to-x})}$ and $j$ are injective.
Indeed, $\psi \circ \mathrm{(\ref{eq:map-to-x})}$ is $u\mapsto (1\ 2)(3\ 4)$, $v\mapsto t$ by Equation~(\ref{eq:psi-conds}) and is thus injective.
Let us use Lemma~\ref{lem:alg-thm2} to show that $j$ is injective.
First,
\[
C_2\{f\}\times C_2\{a\}\times C_2\{v\}\to S_{4}\times C_2\{t\}
\]
maps $f\mapsto(3\ 4)$, $a\mapsto(1\ 2)$, and $v\mapsto t$ by Equation~(\ref{eq:psi-conds}),
and so it is injective.
Second,
\begin{equation}\label{eq:use-t34}
C_2\{f\}\times C_2\{a\}\times C_2\{v\}\to C_2\{g\}\times\pi_{1}(S^{3}\setminus\mathring{N}(K))/\left\langle \left\langle \lambda_{K}^{2},\mu_{K}^{2}\right\rangle \right\rangle
\end{equation}
maps $f\mapsto g$, $a\mapsto \lambda_K$, $v\mapsto \mu_K$ by (\ref{eq:svk-map}),
and so it is injective by Lemma~\ref{lem:t34-injective}.
Hence $j$ is injective by Lemma~\ref{lem:alg-thm2}.
Thus we have shown the bottom row of the diagram in~(\ref{eq:com-diag}) is injective, so we find from the diagram that the map (\ref{eq:want-inj}) is injective.

\vspace{3mm}
\noindent\textbf{Step 4} (Construct $\psi$)\textbf{.}
In this step we construct a map $\psi$ that satisfies Equation~(\ref{eq:psi-conds}), completing the proof of Theorem~\ref{thm:main-pi1-comp}.
Recall that $X$ is the mapping
torus of $r\beta\times r:A\times S_{v}^{1}\to A\times S_{v}^{1}$
and that $\overline{x}=(0,y,0)$.
The fundamental group $\pi_1 (A,y)$ is the free group $F_4$ on four letters $\gamma_{1},\gamma_{2},\gamma_{3},\gamma_{4}$.
Let $\beta,r:F_4 \to F_4$ denote the maps on $F_4 = \pi_1 (A,y)$
induced by the diffeomorphisms $\beta,r:(A,y)\to (A,y)$.
Then
\begin{equation}\label{eq:pi1X-quotient}
\pi_{1}(X,\overline{x})/\left\langle \left\langle [0,1]\times\{x\}\right\rangle \right\rangle \cong\left(F_{4}/\left\langle \left\langle \gamma_{i}^{-1}(r\beta(\gamma_{i}))\ \forall i\right\rangle \right\rangle \right)\times C_2\{v\}.
\end{equation}
Explicitly, this isomorphism (\ref{eq:pi1X-quotient}) maps
\begin{eqenumerate}
    \eqitem\label{enu:gammai}the loop $\{0\}\times \gamma_i \times \{0\}$ to $\gamma_i \in F_{4}/\left\langle \left\langle \gamma_{i}^{-1}(r\beta(\gamma_{i}))\ \forall i\right\rangle \right\rangle$, and
    \eqitem\label{enu:v-iso}the loop $v=\{0\}\times \{y\} \times S_v^1$ to the generator $v\in C_2\{v\}$.
\end{eqenumerate}

Let $F:=\ast_{i=1}^{4} C_2\{\gamma_{i}\}=F_4 / \left\langle \left\langle \gamma_i ^2 \ \forall i\right\rangle \right\rangle$
and denote the composition of $\beta:F_4\to F_4$ (resp.\ $r:F_4\to F_4$) with the quotient $F_4\to F$ also as $\beta:F_4 \to F$ (resp.\ $r:F_4 \to F$).
Let
\begin{equation*}
\Phi: \pi_{1}(X,\overline{x})/\left\langle \left\langle [0,1]\times\{x\}\right\rangle \right\rangle \to\left(F/\left\langle \left\langle \gamma_{i}(r\beta(\gamma_{i}))\ \forall i\right\rangle \right\rangle \right)\times C_2\{v\}
\end{equation*}
be the composition of the map (\ref{eq:pi1X-quotient}) with the quotient map given by quotienting by $\left\langle \left\langle \gamma_{i}^{2} \ \forall i\right\rangle \right\rangle $.
Then, by (\ref{enu:gammai}) and (\ref{enu:v-iso}) we have
\begin{equation}\label{eq:image-Phi}
\Phi(\{0\}\times \gamma_i \times \{0\}) = \gamma_i,\ \Phi(v) = v.
\end{equation}

Now, let us define a map
\begin{equation}\label{eq:psibar}
\overline{\psi}: F/\left\langle \left\langle \gamma_{i}(r\beta(\gamma_{i}))\ \forall i\right\rangle \right\rangle \to S_4.
\end{equation}
First, define $\overline{\psi}:F\to S_{4}$ by letting
\begin{equation}\label{eq:psibar-gamma}
\overline{\psi}(\gamma_{1})=(1\ 2),\ 
\overline{\psi}(\gamma_{2})=(2\ 3),\ 
\overline{\psi}(\gamma_{3})=(2\ 3),\ 
\overline{\psi}(\gamma_{4})=(3\ 4).
\end{equation}
We claim that this map descends to (\ref{eq:psibar}),
i.e.\ that $\overline{\psi}\left(\gamma_{i}^ {}(r\beta(\gamma_{i}))\right)=\mathrm{Id}$
for all $i$.
Let us compute $r\beta(\gamma_{i})$;
recall the action of $\beta$ from Convention \ref{enu:action-gamma} and that $r\gamma_{i}=\gamma_{5-i}$ in $F$.
We have
\begin{align*}
r\beta(\gamma_{1}) & =\gamma_{4}\gamma_{3}\gamma_{4}\gamma_{2}\gamma_{1}\gamma_{2}\gamma_{4}\gamma_{2}\gamma_{1}\gamma_{2}\gamma_{4}\gamma_{3}\gamma_{4}, &
r\beta(\gamma_{2}) & =\gamma_{4}\gamma_{3}\gamma_{4}\gamma_{2}\gamma_{1}\gamma_{2}\gamma_{4}\gamma_{2}\gamma_{1}\gamma_{2}\gamma_{4}\gamma_{2}\gamma_{1}\gamma_{2}\gamma_{4}\gamma_{3}\gamma_{4},\\
r\beta(\gamma_{3}) & =\gamma_{2}, &
r\beta(\gamma_{4}) & =\gamma_{2}\gamma_{4}\gamma_{3}\gamma_{4}\gamma_{2}.
\end{align*}
Now by computing $\overline{\psi}(\gamma_i (r\beta(\gamma_{i})))$ using Equation~(\ref{eq:psibar-gamma}),
we can check that it is the identity for all $i$.
Hence $\overline{\psi}$ descends to~(\ref{eq:psibar}).

Define $\psi:\pi_{1}(X,\overline{x})/\left\langle \left\langle [0,1]\times\{x\}\right\rangle \right\rangle \to S_{4}\times C_2\{t\}$
as the composition of the map $\Phi$ with 
\begin{equation}\label{eq:map-compose}
\left(F/\left\langle \left\langle \gamma_{i}(r\beta(\gamma_{i}))\ \forall i\right\rangle \right\rangle \right)\times C_2\{v\}\to S_{4}\times C_2\{t\}:\gamma_i v^{n}\mapsto\overline{\psi}(\gamma_i )t^{n}.
\end{equation}
We are left to check Equation~(\ref{eq:psi-conds}).
For this, we first study the images of $v,u,a,f$ under $\Phi$ using Equation~(\ref{eq:image-Phi}) and then evaluate the images under $\psi$ using Equation~(\ref{eq:psibar-gamma}).
\begin{itemize}[leftmargin=*]
\item Since $\Phi(v)=v$ and the map (\ref{eq:map-compose}) sends $v$ to $t$, we have $\psi(v)=t$.
\item By Equation~(\ref{eq:extgluing}) we have $u=\{0\}\times \partial D^2 \times \{0\} \subset X$.
Hence $u$ is homotopic to the concatenation of 
$\{0\}\times \gamma_1 \times \{0\}, \cdots , \{0\}\times \gamma_4 \times \{0\}$,
and so 
$\Phi(u)=\gamma_{1}\gamma_{2}\gamma_{3}\gamma_{4}$.
Thus $\psi(u)=(1\ 2)(3\ 4)$.
\item By (\ref{enu:a}) we have $\Phi(a)=\gamma_{1}$, and so $\psi(a)=(1\ 2)$.
\item Let $\delta_{i}$ be the loop in $X$ based at $\overline{x}$ given by the concatenation of $\{0\}\times \rho_{i}\times\{0\}$,
$[0,1]\times\{y_{i}\}\times\{0\}$, and $\{0\}\times \rho_{r\beta(i)}^{-1}\times\{0\}$
(recall from Remark~\ref{rem:transitive} that $r\beta : \{ 1,2,3,4\} \to \{1,2,3,4\}$ is the cycle $(1\ 4\ 3\ 2)$).
Then by (\ref{enu:intf}) and (\ref{enu:f}), $f$ is homotopic in $X$ to $\delta_{1}\delta_{4}\delta_{3}\delta_{2}$.

Let us compute $\Phi(\delta_i)$.
As a loop in $X$ based at $\overline{x}$,
$\delta_{i}$ is homotopic to the concatenation of $[0,1]\times\{x\}=[0,1]\times\{y\}\times\{0\}$,
$\{0\}\times r\beta(\rho_{i})\times\{0\}$, and $\{0\}\times \rho_{r\beta(i)}^{-1}\times\{0\}$.
Hence in $\pi_{1}(X,\overline{x})/\left\langle \left\langle [0,1]\times\{x\}\right\rangle \right\rangle $ (which is the domain of $\Phi$),
$\delta_i$ is equal to the concatenation of
$\{0\}\times r\beta(\rho_{i})\times\{0\}$ and $\{0\}\times \rho_{r\beta(i)}^{-1}\times\{0\}$.
Recall Convention \ref{enu:action-rho}; after taking the quotient  $\gamma_i ^{-1} = \gamma_i$ (to simplify the following expressions), we have

\[
\beta(\rho_{1})=\gamma_{1}\gamma_{2}\gamma_{1}\gamma_{3}\gamma_{4}\gamma_{3}\rho_{1},\ 
\beta(\rho_{2})=\gamma_{1}\gamma_{2}\gamma_{1}\gamma_{3}\gamma_{4}\gamma_{3}\gamma_{1}\gamma_{3}\gamma_{4}\rho_{4},\ 
\beta(\rho_{3})=\rho_{3},\ 
\beta(\rho_{4})=\gamma_{3}\gamma_{1}\rho_{2},
\]
and so we have
\[
\Phi(\delta_{1})=\gamma_{4}\gamma_{3}\gamma_{4}\gamma_{2}\gamma_{1}\gamma_{2},\ 
\Phi(\delta_{2})=\gamma_{4}\gamma_{3}\gamma_{4}\gamma_{2}\gamma_{1}\gamma_{2}\gamma_{4}\gamma_{2}\gamma_{1},\ 
\Phi(\delta_{3})=e,\ 
\Phi(\delta_{4})=\gamma_{2}\gamma_{4}.
\]

Finally, since $f=\delta_1 \delta_4 \delta_3 \delta_2$, we have 
\[
\Phi(f)=\Phi(\delta _1) \Phi (\delta _4 )\Phi (\delta _3 )\Phi (\delta_2 ) 
=\gamma_{4}\gamma_{3}\gamma_{4}\gamma_{2}\gamma_{1}\gamma_{3}\gamma_{4}\gamma_{2}\gamma_{1}\gamma_{2}\gamma_{4}\gamma_{2}\gamma_{1},
\]
and now we can check that $\psi(f)=(3\ 4).$
\end{itemize}
This completes the proof of Theorem~\ref{thm:main-pi1-comp}.
\end{proof}

\subsection{\label{subsec:irreducible-klein}An irreducible Klein bottle in $S^4$ with normal Euler number $\pm 4$}
In this subsection we show that the connected sum $R\# R$ of two copies of our irreducible projective plane $R$ is irreducible, and hence prove Theorem~\ref{thm:klein}. To show that $R \# R$ is irreducible, we use Lemma~\ref{lem:reducible-klein}.

\begin{lem}\label{lem:reducible-klein}
If an embedded Klein bottle $\Sigma \subset S^4$ is the connected sum of a projective plane and an unknotted projective plane,
then the peripheral subgroup of $\Sigma$ is finite.
\end{lem}
\begin{proof}
    Suppose $\Sigma$ is the connected sum of a projective planes $P_1$, $P_2$ with $P_2$ unknotted, so $(S^4 , \Sigma ) = (B_1 ^4 , P_1 ^\circ ) \cup (B_2 ^4 , P_2 ^\circ )$ as pairs of unoriented manifolds,
    where the superscript `$^\circ$' denotes a tangle obtained by deleting a trivial disk.
    Here, we identify the boundaries of $(B^4_1, P_1 ^\circ )$ and $(B^4_2 , P_2 ^\circ)$ with $(S^3 , U)$,
    where $U$ is the unknot,
    and we glue by the identity of $(S^3,U)$.

    Let $N(U)$ be a tubular neighborhood of $U$ in $S^3$.
    Denote as $\partial_n N(P_i ^\circ)$ the subspace of the boundary $\partial N(P_i ^\circ)$ given by the normal circle bundle over $P_i ^\circ$
    (hence $\partial N(P_i^\circ ) = N(U) \cup \partial_n N(P_i^\circ)$).
    Then $\partial N(U)$ is the boundary of $\partial_n N(P_i ^\circ)$ for $i=1,2$,
    and $\partial N(\Sigma) = \partial_n N(P_1 ^\circ) \cup \partial_n N(P_2 ^\circ)$ as unoriented manifolds.
    
    Choose a basepoint $z\in \partial N(U)$, and let $\lambda_U , \mu_U$ be the longitude and meridian of $U$.
    Let us first compute
    \begin{equation}\label{eq:torus-bdrymobius}
        \pi_1 (\partial N(U),z) \to \pi_1 (\partial_n N(P_i ^\circ ), z).
    \end{equation}
    We use that $\partial_n N(P_i ^\circ)$ is the mapping torus of the map $r\times r:[0,1] \times S^1 \to [0,1]\times S^1: (x,y)\mapsto (1-x,1-y)$ (recall we identify $S^1 = [0,1]/(0\sim 1)$);
    under this identification
    the subspace $\partial N(U)$ of  $\partial_n N(P_i ^\circ) $ corresponds to the mapping torus of $\{0, 1\} \times S^1$.
    Hence, the map (\ref{eq:torus-bdrymobius}) is
    \[
     C_\infty \{ \lambda _U \} \times C_\infty \{\mu _U\} \to \left\langle v_i ,  w_i  \middle| w_i v_i w_i ^{-1} v_i \right\rangle: \lambda_U \mapsto w_i ^2 v_i^{k_i},\ \mu _U \mapsto v_i
    \]
    for some $k_i \in \mathbb{Z}$.
    Since we glue $\partial_n N(P_1 ^\circ)$ and $\partial_n N(P_2 ^\circ)$ by the identity of $\partial N(U)$,
    by Seifert--van Kampen we have
    \[
    \pi_1 (\partial N(\Sigma),z) \cong \left\langle v_1 , v_2 , w_1 , w_2 \middle| w_1 v_1 w_1 ^{-1} v_1 , w_2 v_2 w_2 ^{-1} v_2 , v_1 v_2^{-1} , w_1^2 v_1 ^{k_1} (w_2^2 v_2^{k_2})^{-1}\right\rangle .
    \]

    Since $P_2$ is an unknotted projective plane,
    we have $\pi_1 (B_2 ^4 \setminus P_2 ^\circ , z) \cong C_2 \{\mu _U\}$.
    Hence the map
    \begin{equation}\label{eq:mobius-peripheral}
        \pi_1 (\partial_n N(P_2 ^\circ),z)\to \pi_1 (B_2 ^4 \setminus P_2 ^\circ , z) \cong C_2 \{\mu _U\}
    \end{equation}
    sends $v_2 \mapsto \mu _U$, $v_2 ^2 , w_2 ^2  \mapsto 1 $.
    Thus the peripheral map
    \begin{equation}\label{eq:peripheral-klein}
    \pi_1 (\partial N(\Sigma), z) \to \pi_1 (S^4 \setminus \Sigma, z)
    \end{equation}
    factors through $\pi_1 (\partial N(\Sigma),z)/ \left\langle \left\langle v_2 ^2 , w_2^2  \right\rangle \right\rangle $, which is 
    \begin{multline*}
     \left\langle v_1 , v_2 , w_1 , w_2 \middle| w_1 v_1 w_1 ^{-1} v_1 , w_2 v_2 w_2 ^{-1} v_2 , v_1 v_2^{-1} , w_1^2 v_1 ^{k_1} (w_2^2 v_2^{k_2} )^{-1}\right\rangle \bigg/ \left\langle \left\langle v_2 ^2 , w_2^2  \right\rangle \right\rangle \\
    = \left\langle  v_2 , w_1 , w_2 \middle| [w_1 , v_2], [w_2 ,v_2],  w_1^2  v_2^{k_1 + k_2} ,w_2^2 , v_2 ^2 \right\rangle .
    \end{multline*}
    Call this group $G$.
    Since the map (\ref{eq:mobius-peripheral}) maps $v_2$ to $\mu _U$,
    its kernel contains $w_2$ or $v_2 w_2$.
    Hence the map (\ref{eq:peripheral-klein}) factors through
    $G/\left\langle \left\langle w_2 \right\rangle \right\rangle$ or $G/\left\langle \left\langle v_2w_2 \right\rangle \right\rangle$.
    Both are isomorphic to 
    \[
    \left\langle  v_2 , w_1  \middle| [w_1 , v_2],  w_1^2  v_2^{k_1 + k_2} , v_2 ^2 \right\rangle =  C_4\{w_1\} \ \mathrm{or}\ C_2 \{v_2 \} \times C_2 \{w_1 \},
     \]
     depending on the parity of $k_1+k_2$. Both groups are finite, 
    hence the image of the peripheral map (\ref{eq:peripheral-klein}) is finite.
\end{proof}

We will show that the peripheral subgroup of the Klein bottle $\Sigma := R\# R$ is infinite; this implies Theorem~\ref{thm:klein} by Lemma~\ref{lem:reducible-klein}.
Let us first record the following lemma.

\begin{lem}[{\cite[Proposition~3, Section~I.3]{MR607504}}]
\label{lem:amal}
    Let $H,G_{1},G_{2}$ be groups, let $H_i \subset G_i$ be subgroups for $i=1,2$, and let $f_{i}:H\to H_{i}$ be injective homomorphisms for both $i=1,2$.
    Then the induced map $H_{1}\ast_{H}H_{2} \to G_{1}\ast_{H}G_{2}$ is injective.
\end{lem}

We use Lemma~\ref{lem:uv-htpe} to study the peripheral subgroup;
we need to define some objects before we can state it. (These definitions are similar to the definitions in the proof of Lemma~\ref{lem:reducible-klein}.)
Recall the basepoint $\overline{x'}$, the tubular neighborhood $N(P)$, and the loops $u',v'$ from the proof of Theorem~\ref{thm:main-thm} (Figure~\ref{fig:pq-link}~(c)).
Let $(B^4 , P^\circ)$  be obtained by removing the southern hemisphere $\mathring{B}_S^4 \subset S^4$;
the boundary of $P^\circ$ is an unknot $U\subset S^3$.
Let $n_p \subset D^2 \times [-1,1]$ be the neighborhood of $p$ drawn in Figure~\ref{fig:pq-link}~(c).
Let $\partial_n n_p \subset \partial n_p$ be the normal circle bundle over $p$; hence $\partial n_p = \partial_n n_p \cup (n_p \cap (D^2 \times \{-1\}))$.
Let $\partial_n N(P^\circ)$ be the half-spin of $\partial_n n_p$,
and let $N(U)$ be the half-spin of $n_p \cap (D^2\times \{-1\})$.
Then $N(U)$ is a tubular neighborhood of $U$ in $S^3$ and we have $\partial N(P^\circ) = \partial_n N(P^\circ) \cup N(U)$.
The basepoint $\overline{x'}$ and the loop $v'$ lie in $\partial N(U)$,
and in fact $v'$ is the meridian of $U$.

\begin{lem}\label{lem:uv-htpe}
The loops $u'$ and $v'$ can be homotoped in $B^4 \setminus (P^\circ \sqcup \mathring{N}(Q)) $ onto $\partial_n N(P^\circ )$.
\end{lem}
\begin{proof}
    The loop $v'$ is already in $\partial N(U) \subset \partial_n N(P^\circ )$.
    Let us show the statement for $u'$.
    The intersection of $u'$ with $D^2\times \{-1\}$ (i.e.\ the ``horizontal part'' of $u'$ from Figure~\ref{fig:pq-link}~(c)) is the only part of $u'$ that does not already lie in $\partial_n N(P^\circ )$.
    This horizontal part is a path in $S^3 \setminus U$ whose two endpoints are on $\partial N(U)$,
    and any such path can be homotoped rel.\ $\partial$ in $S^3\setminus U$ onto $\partial N(U) \subset \partial_n N(P^\circ )$.
    Hence the lemma follows.
\end{proof}

Recall that $R$ is the image of $P$ in $(S^{4}\setminus\mathring{N}(Q))\cup W \cong S^4$.
Let $R^\circ$ and $\partial_n N(R^\circ)$ be the images of 
$P^\circ$ and $\partial_n N(P^\circ)$, respectively,
in $(B^{4}\setminus\mathring{N}(Q))\cup W \cong B^4$.
    
\begin{proof}[Proof of Theorem~\ref{thm:klein}]
    Let $\Sigma := R\#R$; so $(S^4 , \Sigma )$ is the union of two copies $ (B_1 ^4 , R_1 ^\circ ) $ and $ (B_2 ^4 , R_2 ^\circ )$ of $(B^4 , R^\circ)$,
    glued by an orientation-reversing diffeomorphism of their boundary $(S^3, U)$.
    We have $\overline{x'} \in \partial B_i^4 \setminus R_i^\circ$ for both $i=1,2$.
    Let $u_i ', v_i ',\partial_n N(R_i^\circ)$ be the copies of $u',v',\partial_n N(R^\circ)$, respectively, in $B_i ^4 \setminus R_i ^\circ$.
    By Lemma~\ref{lem:uv-htpe} the loops $u_i',v_i'$ can be homotoped in $B_i^4 \setminus R_i^\circ  $ onto $\partial_n N(R_i ^\circ )$.
    Since $\partial N(\Sigma)  = \partial_n N(R_1 ^\circ )\cup \partial_n N(R_2 ^\circ)$, the loops $u_i',v_i'$ can be homotoped in $S^4 \setminus \Sigma $ onto $\partial N(\Sigma)$.
    Hence, $u_i' , v_i'\in \pi_1 (S^4 \setminus \Sigma , \overline{x'})$ lie in the peripheral subgroup.

    Let us show that the subgroup of $\pi_1 (S^4 \setminus \Sigma , \overline{x'})$ that $u_1 ', u_2 ', v_1 ', v_2 '$ generate is infinite.
    Let $G_i := \pi_1 (B_i ^4 \setminus R^\circ_i , \overline{x'} )$,
    and let $H_i \subset G_i $ be the subgroup generated by $u_i '$ and $v_i '$.
    Then by Theorem~\ref{thm:main-pi1-comp} we have $H_i = C_2 \{u_i '\} \times C_2 \{v_i '\}$.
    The complement $S^4 \setminus \Sigma $ is obtained by gluing $B_1 ^4 \setminus R_1^\circ $ and $B_2 ^4 \setminus R_2^\circ $ along their boundary $S^3 \setminus U$, and so
    by Seifert--van Kampen we have
    \begin{equation*}
    \pi_1 (S^4 \setminus \Sigma,\overline{x'})
    \cong G_1 \ast _{C_\infty \{s\}} G_2,
    \end{equation*}
    where $C_\infty \{s\}\to G_i $ is given by $s\mapsto 
    {v_i '}^{\pm 1}$
    (recall that $v_i'$ is the meridian of $U$ in $S^3$).
    Since $v_i '^2$ is trivial in $G_i$, we have
    \[
    G_1 \ast _{C_\infty \{s\}} G_2 = G_1 \ast _{C_2 \{s\}} G_2.
    \]
    Hence, by Lemma~\ref{lem:amal},
    \[
    H_1 \ast _{C_2 \{s\}} H_2 \to G_1 \ast _{C_2 \{s\}} G_2 
    \]
    is injective.
    Moreover,
    \[H_1 \ast _{C_2 \{s\}} H_2  = (C_2 \{ u_1 '\} \ast C_2 \{ u_2 '\} )\times C_2 \{v_1 '\},\]
    where $v_1 '$ and $v_2 '$ are identified;
    hence $H_1 \ast _{C_2 \{v'\}} H_2$ is infinite.
    This group is generated by $u_1 ', u_2 ', v_1 ', v_2 '$ and is thus contained in the peripheral subgroup.
    Hence the peripheral subgroup is infinite and so Theorem~\ref{thm:klein} follows from Lemma~\ref{lem:reducible-klein}.
\end{proof}

\begin{rem}
    The normal Euler number of the Klein bottle is irrelevant to the proofs of this subsection.
    In particular, the same arguments show that $R\# (-R)$ (which has normal Euler number $0$) is an irreducible Klein bottle.
    Indeed, Lemma~\ref{lem:reducible-klein} is for Klein bottles of any normal Euler number.
    Furthermore, the proof of Theorem~\ref{thm:klein} also works for the Klein bottle $R\# (-R)$:
    the only difference is that now we glue $ (B_1 ^4 , R_1 ^\circ ) $ and $-(B_2 ^4 , R_2 ^\circ )$,
    and the computations are the same.
    Hence, the peripheral subgroup for $R\# (-R)$ is also infinite.
\end{rem}

\bibliographystyle{amsalpha}
\bibliography{bib}

\end{document}